\documentclass[11pt, leqno]{amsart}
\usepackage{graphicx}
\usepackage{amsfonts,delarray,amssymb,amsmath,amsthm,a4,a4wide}
\usepackage{latexsym}
\usepackage{epsfig}
\usepackage{color}
\usepackage[margin=1in]{geometry} 

\newtheorem{thm}{Theorem}[section]
\newtheorem{cor}[thm]{Corollary}

\newtheorem{lem}[thm]{Lemma} 
\newtheorem{exam}[thm]{Example}
\newtheorem{prop}[thm]{Proposition}
\newtheorem{rem}[thm]{Remark}

\theoremstyle{definition}

\numberwithin{equation}{section}

\newcommand{\R}{\mathbb R}

\newcommand{\e}{\varepsilon}

\newcommand{\p}{\partial}

\newcommand{\trace}{\mbox{trace}\,}

\newcommand{\comment}[1]{}
\def\h{\hspace*{.24in}}

\newenvironment{myindentpar}[1]%
{\begin{list}{}%
         {\setlength{\leftmargin}{#1}}%
         \item[]%
}
{\end{list}}
\makeatletter
\@namedef{subjclassname@2020}{%
  \textup{2020} Mathematics Subject Classification}
\makeatother
\begin{document} 

\title[Spectral characterization and approximation for the Hessian eigenvalue]{A spectral characterization and an approximation scheme for the Hessian eigenvalue }
\author{Nam Q. Le}
\address{Department of Mathematics, Indiana University, 831 E 3rd St,
Bloomington, IN 47405, USA}
\email{nqle@indiana.edu}
\thanks{The research of the author was supported in part by the National Science Foundation under grant DMS-2054686.}



\maketitle
\begin{abstract}
We revisit the $k$-Hessian eigenvalue problem on a smooth, bounded, $(k-1)$-convex domain in $\R^n$. First, we obtain a spectral characterization of the $k$-Hessian eigenvalue as the infimum of the first eigenvalues of linear second-order elliptic operators  whose coefficients belong to the dual of the corresponding G\r{a}rding cone. Second, 
we introduce a non-degenerate inverse iterative scheme to solve the eigenvalue problem for the $k$-Hessian operator. 
We show that the scheme converges, with a rate, to the  $k$-Hessian eigenvalue for all $k$. When $2\leq k\leq n$, we also prove a local $L^1$ convergence of the Hessian of solutions of the scheme. 
Hyperbolic polynomials play an important role in our analysis.
\end{abstract}

\section{Introduction and statements of the main results} 
In this paper, we consider the $k$-Hessian counterparts of some results on  the Monge-Amp\`ere eigenvalue problem. We begin by recalling these results and relevant backgrounds.
\subsection{The Monge-Amp\`ere eigenvalue problem}
The Monge-Amp\`ere eigenvalue problem on smooth, bounded and uniformly convex domains $\Omega$ in $\R^n$ ($n\geq 2$) was 
first investigated by Lions \cite{Lions}. He showed that there exist a unique positive constant $\lambda=\lambda(n;\Omega)$ and a unique 
(up to positive multiplicative constants) nonzero 
convex function $u\in C^{1,1}(\overline{\Omega})\cap C^{\infty}(\Omega)$ solving the eigenvalue problem for the Monge-Amp\`ere operator
$\det D^2 u$: 
 \begin{equation}\label{EVP_eq}
\det D^2 u~ = \lambda^n|u|^n~\text{in} ~ \Omega~\text{and}~
u = 0~\text{on}~ \p\Omega.
  \end{equation}
The constant $\lambda(n;\Omega)$ is called 
the Monge-Amp\`ere eigenvalue of $\Omega$. 
The functions $u$ solving (\ref{EVP_eq}) are called the Monge-Amp\`ere eigenfunctions. Lions also found a spectral characterization of Monge-Amp\`ere eigenvalue via the first eigenvalues of linear second-order elliptic operators in non-divergence form.

 Let $V_n=V_n(\Omega)$ be the set of all matrices $A=(a_{ij})_{1\leq i, j\leq n}$ with $a_{ij}\in C(\Omega)$,
 $$(a_{ij})= (a_{ji})>0\quad\text{in }\Omega,\quad \text{and }  \det (A) \geq \frac{1}{n^n}.$$
 
 For $A\in V_n$, let $\lambda_1^A$ be the first (positive) eigenvalue of the linear second order operator $-a_{ij} D_{ij}$ with zero Dirichlet boundary condition on $\p\Omega$ so there exist $v\in W^{2, n}_{loc}(\Omega)\cap C(\overline{\Omega})$  with $v\not\equiv 0$ such that
 \begin{equation*}
 -a_{ij} D_{ij} v= \lambda_1^A v\quad\text{in }\Omega,\quad v=0\quad \text{on }\p\Omega.
 \end{equation*}
   The corresponding eigenfunctions $v$, up to multiplicative constants, are positive in $\Omega$ and unique. We refer the readers to the Appendix in \cite{Lions} for more information about the first eigenvalues for  $-a_{ij} D_{ij}$ where $A\in V_n$.  Lions \cite{Lions} showed that

 \begin{equation}
 \label{Llam1}
 \lambda(n; \Omega) =\min_{A\in V_n} \lambda^A_1.
 \end{equation}

A variational characterization of $\lambda(n; \Omega)$  was first discovered by Tso \cite{Tso}.
Denote the Rayleigh quotient (for the Monge-Amp\`ere operator) of a nonzero convex function $u$ by $$R_n(u) = \frac{\int_{\Omega} |u|\det D^2 u~dx}{\int_{\Omega} |u|^{n+1}~dx}.$$ 
When $u$ is merely a convex function, $\det D^2 u~ dx$ is interpreted as the Monge-Amp\`ere measure associated with $u$; see Figalli \cite{F2} and Guti\'errez \cite{G01}. 
Tso showed that 
\begin{equation}
\label{lam_def1}
 [\lambda(n; \Omega)]^n =\inf\left\{ R_n(u): u\in C^{0,1}(\overline{\Omega})\cap C^{\infty}(\Omega),
 ~u~\text{is convex, nonzero in } \Omega,~ u=0~\text{on}~\p\Omega\right\}.
 \end{equation}
 \vglue 0.2cm
Recently, the author \cite{LSNS} studied the Monge-Amp\`ere eigenvalue problem for general 
open bounded convex domains and established the singular counterparts of previous results by  Lions and Tso. 
 Let $\Omega$ be a bounded open convex domain in $\R^n$. Define the constant $\lambda=\lambda[n; \Omega]$ via infimum of the Rayleigh quotient by
\begin{equation}
\label{lam_def}
 (\lambda[n; \Omega])^n =\inf\left\{ R_n(u): u\in C(\overline{\Omega}),
 ~u~\text{is convex, nonzero in } \Omega,~ u=0~\text{on}~\p\Omega\right\}.
 \end{equation}
 
Then, by \cite{LSNS}, the infimum in (\ref{lam_def}) is achieved
 because there exists a 
 nonzero convex eigenfunction $u\in C(\overline{\Omega})\cap C^{\infty}(\Omega)$ solving the Monge-Amp\`ere eigenvalue problem (\ref{EVP_eq}) with $\lambda=\lambda[n; \Omega]$.
 When $\Omega$ is a smooth, bounded and uniformly convex domain, the class of competitor functions in the minimization problem (\ref{lam_def}) is larger than that of the minimization problem (\ref{lam_def1}); however, it was shown in \cite{LSNS} that $\lambda(n;\Omega)= \lambda [n;\Omega].$

In \cite{AK}, Abedin and Kitagawa introduced a numerically appealing inverse iterative scheme
\begin{equation}\label{IIS}
\det D^2u_{m+1} = R_n(u_m) |u_m|^n  \quad \text{in } \Omega, ~
u_{m+1} = 0  \quad \text{on } \partial \Omega
\end{equation}
 to solve 
 the Monge-Amp\`ere eigenvalue problem (\ref{EVP_eq})
  on a bounded convex domain $\Omega\subset\R^n$. They proved that the scheme (\ref{IIS}) converges to the Monge-Amp\`ere eigenvalue problem (\ref{EVP_eq}) for all convex initial data $u_0$ satisfying $R_n(u_0) < \infty$, $u_0\leq 0$ on $\p\Omega$, and
$\det D^2 u_0 \geq 1$ in $\Omega$. When $m\geq 1$, (\ref{IIS}) is a degenerate Monge-Amp\`ere equation for $u_{m+1}$ because the right hand side tends to $0$ near  the boundary $\p\Omega$.

 In this paper, we prove  a spectral characterization of the $k$-Hessian eigenvalue similar to (\ref{Llam1}) (Theorem \ref{kHessL}), and study a non-degenerate inverse iterative scheme (\ref{kIS}), similar to (\ref{IIS}), to solve the $k$-Hessian eigenvalue problem. We will review this problem in Section \ref{kEVP_sec}.  The main results concerning the scheme (\ref{kIS}) include convergence to the $k$-Hessian eigenvalue (Theorem \ref{conRk}) and local $W^{2,1}$ type convergence (Theorem \ref{W21k}).
 The common thread in our investigation is hyperbolic polynomials to be reviewed in Section \ref{Hypersect}. Our approach, which is based on certain integration by parts inequalities, differs from \cite{AK} even in the Monge-Amp\`ere case. As an illustration, for the Monge-Amp\`ere case, our approach gives a sharp reverse Aleksandrov estimate for the Monge-Amp\`ere equation and  a convergence rate of $R_n(u_m)$ to $(\lambda(n;\Omega))^n$ in terms of the convergence rate of $u_m$ to a nonzero Monge-Amp\`ere eigenfunction $u_{\infty}$ (see, Theorem \ref{conRk} $(ii, iii)$).
  This is new compared to currently known   iteration schemes for the $p$-Laplace equation \cite{BEM, Boz, HL}.

\subsection{The $k$-Hessian eigenvalue problem}
\label{kEVP_sec}
Let $1\leq k\leq n$ ($n\geq 2$).
 Let $\Omega$ be a bounded open and smooth domain in $\R^n$.
 For a function $u\in C^2(\Omega)$, let $S_k(D^2 u)$ denote the $k$-th elementary symmetric function of the eigenvalues $\lambda(D^2 u)=(\lambda_1(D^2 u), \cdots,\lambda_n(D^2 u))$ of the Hessian matrix $D^2 u$:
$$S_k(D^2 u)=\sigma_k(\lambda(D^2 u)):=\sum_{1\leq i_1<\cdots<i_k\leq n}\lambda_{i_1}(D^2 u)\cdots \lambda_{i_k}(D^2 u).$$
For convenience, we denote $\sigma_0(\lambda)=1.$
A function $u\in C^2(\Omega)\cap C(\overline{\Omega})$ is called $k$-admissible if $\lambda(D^2 u)\in \Gamma_k$
where $\Gamma_k$ is an open symmetric convex cone in $\R^n$, with vertex at the origin, given by
\begin{equation}
\label{Gak}
\Gamma_k=\{\lambda=(\lambda_1, \cdots,\lambda_n)\in \R^n\mid \sigma_j(\lambda)>0\quad\forall j=1, \cdots, k\}.
\end{equation}
We also call $\Gamma_k$ the G\r{a}rding cone of the $k$-Hessian operator. 
{\it All functions involved in $S_k$ below are assumed to be $k$-admissible. }
 If $k\geq 2$, we also assume $\p\Omega$ to be uniformly $(k-1)$-convex, that is, $\sigma_{k-1}(\kappa_1,\cdots,\kappa_{n-1})\geq c_0>0$ where $\kappa_1,\cdots,\kappa_{n-1}$ are principle curvatures of $\p\Omega$ relative to the interior normal.  Note that $n$-admissible functions are strictly convex, and uniformly $(n-1)$-convex domains are simply uniformly convex domains.

The eigenvalue problem for the $k$-Hessian operator
$S_k(D^2 u)$
 on a bounded, open, smooth and $(k-1)$-convex domain $\Omega$ in $\R^n$  
 \begin{equation}
 \label{kEVP_eq}
   S_k (D^2 w)~=[\lambda(k;\Omega)]^k |w|^{k} \h~\text{in} ~\Omega,~
w =0\h~\text{on}~\p \Omega
\end{equation}
was first introduced by Wang in \cite{W1} (see also \cite{W2}) who extended the results of Lions \cite{Lions} and Tso \cite{Tso} from the case $k=n$ to the general case $1\leq k\leq n$.
Wang introduced the constant
\begin{equation}
\label{Wscheme}
\lambda_1=\sup\{\lambda>0; \text{there is a solution } u_{\lambda}\in C^2(\overline{\Omega}) \text{ of } (\ref{Wappro})\}
\end{equation}
where (\ref{Wappro}) is given by
\begin{equation}
\label{Wappro} S_k(D^2 u)=(1-\lambda u)^k~\text{in }\Omega,\quad u=0\text{ on }\p \Omega.
\end{equation}
Wang \cite{W1} showed that $\lambda_1\in (0,\infty)$, and that as $\lambda\rightarrow\lambda_1$, $u_{\lambda}\|u_{\lambda}\|^{-1}_{L^{\infty}(\Omega)}$ converges in $C^{\infty}(\Omega)\cap C^{1,1}(\overline{\Omega})$ to a solution $w\in C^{\infty}(\Omega)\cap C^{1,1}(\overline{\Omega})$ of (\ref{kEVP_eq}) with $\lambda(k;\Omega)=\lambda_1$ there. 

The eigenvalue problem (\ref{kEVP_eq}) has the following uniqueness property: if $(\bar \lambda,\bar w)$ solves (\ref{kEVP_eq}), where $\bar\lambda\geq 0$, $\bar w\in C^{\infty}(\Omega)\cap C^{1,1}(\overline{\Omega})$ is k-admissible with $w=0$ on $\p \Omega$, then $\bar \lambda=\lambda (k;\Omega)$
and $\bar w=c w$ for some positive constant $c$. The constant $\lambda(k;\Omega)$ is called the $k$-Hessian eigenvalue and $w$ in (\ref{kEVP_eq}) is called a $k$-Hessian eigenfunction. The scheme (\ref{Wscheme})-(\ref{Wappro}) to compute the $k$-Hessian eigenvalue involves solving the $k$-Hessian equation (\ref{Wappro}) with right hand side depending on the solution $u$ itself. This equation is more difficult to handle, analytically and numerically, than one with right hand side depending only on the spatial variables.

Let $R_k(u)$ denote the Rayleigh quotient for the $k$-Hessian operator
\begin{equation}
\label{RQk}
R_k(u) = \frac{\int_{\Omega} |u| S_k(D^2 u)~dx}{\int_{\Omega} |u|^{k+1}~dx}
\end{equation}
 for a $C^2$ function $u$. 
Implicit in the definition (\ref{RQk}) is the requirement that 
$\|u\|_{L^{k+1}(\Omega)}<\infty.$

Wang \cite{W1} also proved the following fundamental property for the variational characterization of $\lambda(k;\Omega)$: 
\begin{equation}
\label{klamR}
\small
 [\lambda(k;\Omega)]^k =\inf\left\{ R_k(u): u\in C(\overline{\Omega})\cap C^2(\Omega),\text{u is }k\text{-admissible, nonzero in } \Omega, u=0~\text{on}~\p\Omega\right\}.
 \end{equation}
 Using (\ref{klamR}), Liu-Ma-Xu \cite{LMX} obtained a Brunn-Minkowski inequality for the 2-Hessian eigenvalue in three-dimensional convex domains.

\subsection{A spectral characterization of the $k$-Hessian eigenvalue}
 Let $x\cdot y$ denote the standard inner product for $x, y\in\R^n$.
Following Kuo-Trudinger \cite{KT}, let $\Gamma_k^{\ast}$ be the dual cone of the G\r{a}rding cone $\Gamma_k$, given by
 \begin{equation}
 \label{Gaka}
 \Gamma_k^{\ast} =\{\lambda\in\R^n\mid \lambda\cdot\mu \geq0\quad\forall \mu\in\Gamma_k\}.
 \end{equation}
Clearly
 $\Gamma_k^\ast\subset\Gamma_l^\ast\quad\text{for } k\leq l.$
 For $\lambda\in\Gamma_k^{\ast}$, denote
 $$\rho_k^{\ast}(\lambda)=\inf\left\{\frac{\lambda\cdot\mu}{n} \mid \mu\in\Gamma_k, S_k(\mu)\geq {n \choose k}\right\}.$$
 Observe that $\Gamma^\ast_n=\overline{\Gamma_n}$. If $\lambda= (\lambda_1,\cdots,\lambda_n)\in \Gamma_n^\ast$ then $\lambda_i\geq 0$ and $\rho_n^\ast(\lambda)= (\prod_{i=1}^n \lambda_i)^{1/n}$.
 \vglue 0.1cm
 For a matrix $A=(a_{ij})_{1\leq i,j\leq n}$, we write $A\in \Gamma^{\ast}_k$ if $\lambda(A)\in \Gamma^{\ast}_k$ and define $$\rho_k^\ast(A)=\rho_k^\ast(\lambda(A)).$$
 Let $V_k=V_k(\Omega)$ be the following set of positive definite symmetric matrices whose entries are continuous functions on $\Omega$:
 \begin{equation}
 \label{Vkeq}
 \small
 V_k=\left \{A=(a_{ij})_{1\leq i, j\leq n},~  (a_{ij})= (a_{ji})>0\text{ in }\Omega,~ a_{ij}\in C(\Omega),~ A\in \Gamma_k^\ast,~ \text{and }  \rho_k^\ast(A) \geq \frac{1}{n}{n \choose k}^{1/k}\right\}.
 \end{equation}
 Note that 
 $V_k\subset V_{k+1}.$
 This follows from the Maclaurin inequalities and $c(n, k)> c(n, k+1)$ for all $k\leq n-1$ where $c(n, k):=\frac{1}{n}{n \choose k}^{1/k}$.
 Indeed, suppose $A\in V_k$. If $\mu\in \Gamma_{k+1}$, with $S_{k+1}(\mu) \geq {n\choose k+1}$, then from Maclaurin's inequality, 
 $$\left(\frac{S_k(\mu)}{{n\choose k}}\right)^{\frac{1}{k}}\geq \left(\frac{S_{k+1}(\mu)}{{n\choose k+1}}\right)^{\frac{1}{k+1}},$$
 we find that $\mu\in \Gamma_k$ with  $S_k(\mu) \geq {n\choose k}.$ Hence, 
 $\frac{\lambda(A)\cdot\mu}{n}\geq \rho_k^\ast(A)\geq c(n, k)>c(n, k+1),$
 and therefore $A\in V_{k+1}$. That $c(n, k)>c(n, k+1)$ follows from
 \begin{eqnarray*}\left[\frac{c(n, k)}{c(n, k+1)}\right]^{k(k+1)}={n\choose k} \left[{n\choose k}{n\choose k+1}^{-1}\right]^k&=&{n\choose k}  \frac{(k+1)^k}{(n-k)^k}\\&=& \frac{n(n-1)\cdots (n-k+1)}{(n-k)^k} \frac{(k+1)^k}{k!}>1. 
 \end{eqnarray*}
  We now have the following increasing sequence of cones:
 $$V_1\subset V_2\subset\cdots \subset V_{n-1}\subset V_n.$$
 Extending Lions' result (\ref{Llam1}) from $k=n$ to all other values of $k$, we have  the following theorem.
\begin{thm}[A spectral characterization for the Hessian eigenvalue] 
\label{kHessL}
Assume $1\leq k\leq n$.
 Let $\Omega$ be a bounded, open, smooth, and uniformly convex domain in $\R^n$. Let $V_k$ be as in (\ref{Vkeq}). For $A\in V_k$, let $\lambda_1^A$ be the first positive eigenvalue of the linear second order operator $-a_{ij} D_{ij}$ with zero Dirichlet boundary condition on $\p\Omega$. Then
 \begin{equation}
 \lambda(k;\Omega) =\min_{A\in V_k} \lambda^A_1.
 \end{equation}
\end{thm}
The interest of the above theorem is when $k\geq 2$. When $k=1$, we have
$$V_1=\{m I_n,~m\geq 1\}$$
where $I_n$ is the identity $n\times n$ matrix and thus the conclusion of Theorem \ref{kHessL} is obvious.

\subsection{A non-degenerate inverse iterative scheme for the $k$-Hessian eigenvalue problem}
Inspired by the scheme (\ref{IIS}), we propose the following non-degenerate inverse iterative scheme, to solve 
 the eigenvalue problem (\ref{kEVP_eq}), starting from  a $k$-admissible function $u_0\in C^2(\overline{\Omega})$ 
 with $u_0\leq 0$ on $\p\Omega$
\begin{equation}\label{kIS}
S_k (D^2u_{m+1}) = R_k(u_m) |u_m|^k + (m+1)^{-2}  \quad \text{in } \Omega, ~
u_{m+1} = 0  \quad \text{on } \partial \Omega.
\end{equation}
We add the positive constant $(m+1)^{-2}$, which vanishes in the limit $m\rightarrow\infty$, to make the right hand side of (\ref{kIS}) strictly positive for each $m$. Thus, for each $m\geq 0$, (\ref{kIS}) is a non-degenerate $k$-Hessian equation for $u_{m+1}$. See also Remark \ref{remam}.
The requirement $u_0\leq 0$ on $\p\Omega$ is only used to have $u_0\leq 0$ in $\Omega$ and thus $R_k(u_0) |u_0|^k= R_k(u_0)(-u_0)^k \in C^2(\overline{\Omega})$.
  \vglue 0.1cm
 By a classical result of Caffarelli-Nirenberg-Spruck \cite[Theorem 1]{CNS2}  (see also \cite[Theorem 3.4]{W2}), for each $m$, (\ref{kIS}) has a unique $k$-admissible solution $u_{m+1}\in C^{3,\alpha}(\overline{\Omega})$ for all $0<\alpha<1$. Moreover, $u_m< 0$ in $\Omega$ for all $m\geq 1$.
 The sequence $(u_m)$ is obtained by repeatedly inverting the $k$-Hessian operator with Dirichlet boundary condition.

 In the next theorem,  we show that $R(u_m)$ converges to $[\lambda(k;\Omega)]^k$, thus making the scheme (\ref{kIS}) more appealing for numerically computing the $k$-Hessian eigenvalue $\lambda(k;\Omega)$.
 \begin{thm}[Convergence to the Hessian eigenvalue of the non-degenerate inverse iterative scheme]
 \label{conRk}
 Let $1\leq k\leq n$ where $n\geq 2$. Let $\Omega$ be a bounded, open, smooth domain in $\R^n$. Assume that $\p\Omega$ is uniformly $(k-1)$-convex if $k\geq 2$. 
Consider the reverse iterative scheme (\ref{kIS}) where $u_0\in C^{2}(\overline{\Omega})$ with $u_0\leq 0$ on $\p\Omega$, and $u_m$ is $k$-admissible for all $m\geq 0$.  Let $w\in C^{\infty}(\Omega)\cap C^{1,1}(\overline{\Omega})$ be a nonzero $k$-Hessian eigenfunction as in (\ref{kEVP_eq}). Then
\begin{enumerate}
\item [(i)] $R_k(u_m)$ converges to  $[\lambda(k;\Omega)]^k$:
 \begin{equation}
\label{Rku_con}
\lim_{m\rightarrow\infty} R_k(u_m)= [\lambda(k;\Omega)]^k.
\end{equation}
\item[(ii)] There exists a subsequence $u_{m_j}$ that converges weakly in $W^{1, q}_{loc}(\Omega)$ for all $q<\frac{nk}{n-k}$ to a nonzero function $u_{\infty}\in W^{1, q}_{loc}(\Omega)\cap L^{k}(\Omega)$. Moreover, 
\begin{equation}
\label{umLk}
\lim_{j\rightarrow\infty} \int_\Omega |u_{m_j}|^k~dx= \int_\Omega |u_{\infty}|^k~dx
\end{equation}
and, for all $m\geq 1$, 
\begin{equation}
\label{umRate}
R_k^{1/k}(u_m)-\lambda(k;\Omega) \leq  \lambda(k;\Omega)\frac{  \int_{\Omega}(|u_{m+1}|- |u_{m}|) |w|^k~dx} {\int_{\Omega} |u_1| |w|^k~dx}\leq  \lambda(k;\Omega)\frac{  \int_{\Omega}(|u_{\infty}|- |u_{m}|) |w|^k~dx} {\int_{\Omega} |u_1| |w|^k~dx}. 
\end{equation}
\item[(iii)] When $k=n$, $\{u_m\}$ converges uniformly on $\overline{\Omega}$ to a non-zero Monge-Amp\`ere eigenfunction $u_{\infty}$ of $\Omega$. 
\item[(iv)] When $k=1$, $\{u_m\}$ converges in $W_0^{1, 2}(\Omega)$ to a non-zero first Laplace eigenfunction $u_{\infty}$ of $\Omega$.
\end{enumerate}
  \end{thm}
  We point out that part $(iv)$ of Theorem \ref{conRk} was included for completeness, as it was contained in \cite{BEM, Boz, HL} when there is no term $(m+1)^{-2}$ on the right hand side of (\ref{kIS}).
  
In the convex case when $k=n$, in view of the work \cite{LSNS}, the Monge-Amp\`ere eigenvalue problem (\ref{EVP_eq}) with $u$ only being convex (so less regular) is now well understood and this plays a key role in the proof of Theorem \ref{conRk} $(iii)$.
  The work \cite{LSNS} relies on the regularity theory
of weak solutions to the Monge-Amp\`ere equation developed by Caffarelli \cite{C1, C2}.  
To the best of the author's knowledge, for $2\leq k\leq n-1$, the $k$-Hessian counterparts of these Monge-Amp\`ere results are still lacking. Thus, 
showing that $u_{\infty}$ in Theorem \ref{conRk} $(ii)$ is a $k$-Hessian eigenfunction is still an interesting open problem.
 One possible alternate route is to upgrade the convergence of $u_m$ to $u_{\infty}$ in $W_{loc}^{1, q}(\Omega)$ to that  in $W^{2, p}_{loc}(\Omega)$ for some $p>k$. So far, we can prove a sort of local $W^{2,1}(\Omega)$ convergence. It is in fact a local $W^{2,1}(\Omega)$ convergence when $k=n$ (see also Theorem \ref{W21n}). We have the following theorem.
\begin{thm} [Local $W^{2,1}$ convergence of the non-degenerate inverse iterative scheme]
\label{W21k}
Assume $2\leq k\leq n$.
 Let $\Omega$ be a bounded, open, smooth, and uniformly $(k-1)$ convex domain in $\R^n$.
 Let $w\in C^{\infty}(\Omega)\cap C^{1,1}(\overline{\Omega})$ be a nonzero $k$-Hessian eigenfunction as in (\ref{kEVP_eq}).
Consider the scheme (\ref{kIS}) where $u_0\in C^{2}(\overline{\Omega})$ with $u_0\leq 0$ on $\p\Omega$, and $u_m$ is $k$-admissible for all $m\geq 0$. Consider a subsequence of $(u_{m_j})$ and its limit $u_{\infty}$ as in Theorem \ref{conRk} (ii). Let $\lambda_{k, i}(D^2 w, D^2 u_{m+1})$ be defined by
$$S_k(tD^2 w+ D^2 u_{m+1})=S_k(D^2 w)\prod_{i=1}^k (t+ \lambda_{k, i}(D^2 w, D^2 u_{m+1}))\quad \text{for all } t\in \R.$$
When $k=n$, $\lambda_{k, i}(D^2 w, D^2 u_{m+1})$'s are eigenvalues of $D^2 u_{m+1} (D^2 w)^{-1}$.
Then
$$\lambda_{k, i}(D^2 w, D^2 u_{m_j+1}) \rightarrow \frac{|u_{\infty}|}{|w|} \quad\text{locally in } L^1 \text{ when } j\rightarrow\infty.$$
Up to a further extraction of a subsequence, we have the following pointwise convergence:
\begin{equation}
\label{conpt}
D^2 u_{m_j+1}(x) \rightarrow \frac{|u_\infty(x)|}{|w(x)|} D^2 w(x)\quad\text{a.e. }x\in\Omega.
\end{equation}
\end{thm}
\begin{rem}
\label{remam}
The conclusions of Theorems \ref{conRk} and \ref{W21k} hold if we replace $(m+1)^{-2}$ in the scheme (\ref{kIS}) by $a_m>0$ where $\sum_{m=0}^{\infty} a_m<\infty$. When $k=n$, we can also take $a_m=0$, and in this case, (i) and (iii) of Theorem \ref{conRk} were obtained in \cite{AK} with a different proof.
\end{rem}

We now say a few words about the proofs of Theorems \ref{conRk} and \ref{W21k}. 
 When $k<n$, the lack of convexity of $k$-admissible functions is the main difficulty in the proof of Theorem \ref{conRk}. 
 Our approach is based on the following nonlinear integration by parts inequality for the $k$-Hessian operator. 
 
 \begin{prop}[Nonlinear integration by parts inequality for the $k$-Hessian operator] 
 \label{ibplem}
  Let $\Omega$ be a bounded, open, smooth domain in $\R^n$. Assume that $\p\Omega$ is uniformly $(k-1)$-convex if $k\geq 2$. Then,
 for $k$-admissible functions $u, v\in C^{1,1}(\overline{\Omega})\cap C^3(\Omega)$ with $u=v=0$ on $\p\Omega$, one has
\begin{equation}
\label{NIBP}
\int_\Omega |v| S_k(D^2 u)dx \geq \int_\Omega |u| [S_k(D^2 u)]^{\frac{k-1}{k}} [S_k(D^2 v)]^{\frac{1}{k}} dx.
\end{equation}
If $k\geq 2$ and  the equality holds in (\ref{NIBP}), then there is a positive, continuous function $\mu$ such that $$D^2 u(x)= \mu(x) D^2 v(x) \quad \text{ for all }x\in \Omega.$$
\end{prop}
 The Monge-Amp\`ere case of (\ref{NIBP}), that is, when $k=n$ and $u$ and $v$ are convex, was established in \cite{LSNS} under more relaxed conditions on $u, v$ and $\Omega$. 
 
 We will prove Proposition \ref{ibplem}, and its extensions, using G\r{a}rding's inequality \cite{Garding} for  hyperbolic polynomials of which $\sigma_k$'s and $S_k$'s (viewed as functions of matrices) are examples.  
    
 For the proof of Theorem \ref{W21k}, we find that quantitative forms of (\ref{NIBP}) whose defects measure certain closeness of $D^2 u$ to $D^2 v$  guarantee  the interior $W^{2,1}$ convergence of $u_m$ to $u_{\infty}$. They are proved using 
quantitative G\r{a}rding's inequalities for  hyperbolic polynomials; see Lemma \ref{Glem}.

\begin{rem}
When $k=1$, (\ref{NIBP}) becomes an equality and it is an integration by parts formula. If we just require that $\lambda (D^2 u), \lambda (D^2 v)\in \overline{\Gamma_k}$ instead of $\lambda (D^2 u), \lambda (D^2 v)\in \Gamma_k$, then (\ref{NIBP}) still holds.
To see this, take a $k$-admissible function $w\in  C^{1,1}(\overline{\Omega})\cap C^3(\Omega)$ with $w=0$ on $\p\Omega$. Then, we apply the current version of (\ref{NIBP}) to $u+\e w$ and $v+\e w$ and then let $\e\rightarrow 0$.
\end{rem}

The rest of the paper is organized as follows. In Section \ref{Hypersect}, we recall some basics of hyperbolic polynomials and G\r{a}rding's inequality. In Section \ref{ibpksect}, we prove Proposition \ref{ibplem} and its extensions to other hyperbolic polynomials. In Section \ref{dual_sect}, we prove Theorem \ref{kHessL}. In Section \ref{conRksect}, we prove Theorem \ref{conRk}.
The proof of Theorem \ref{W21k} will be given in Section \ref{W21sect}.
 \section{Hyperbolic polynomials }
 \label{Hypersect}
 In this section, we recall some basics of hyperbolic polynomials and G\r{a}rding's inequality. See also Harvey-Lawson \cite{HLn} for a simple and self-contained account of G\r{a}rding's theory of hyperbolic polynomials \cite{Garding}.

 Suppose that $p$ is a homogenous real polynomial of degree $k$ on $\R^N$. Given $a\in \R^N$, we say that $p$ is {\it $a$-hyperbolic} if $p(a)>0$ and for each $x\in\R^N,$ $p(ta + x)$ can be factored as
 $$p(ta + x) =p(a) \prod_{i=1}^k (t + \lambda_i(p; a,x))\quad \text{for all } t\in \R$$
 where $\lambda_i(p; a, x)$'s ($i=1, \cdots, k$) are real numbers. The functions $\lambda_i(p; a, x)$ are called the {\it $a$-eigenvalues of $x$}, and they are well-defined up to permutation. In what follows, identities between $\lambda_i(p;\cdot, \cdot)$ are understood modulo the permutation group $\mathcal{S}_k$ of order $k$. 
 
 For reader's convenience, we mention here some examples of $a$-hyperbolic polynomials, mostly taken from \cite{Garding}. The polynomials $P_k$'s in Example \ref{Ex3} are most relevant for the results of this paper.

 \begin{exam}
 The quadratic polynomial $$p(x)= x_1^2-x_2^2-\cdots-x^2_N,\quad x=(x_1, \cdots, x_N)\in \R^N$$
 is $e_1$-hyperbolic where $e_1=(1,0,\cdots,0)$. The $e_1$-eigenvalues of $x\in\R^N$ are given by
 $$\displaystyle \left\{\lambda_1(p; e_1, x), \lambda_2(p; e_1, x)\right\}=\left\{x_1\pm \sqrt{|x|^2-x_1^2}\right\}.$$
 \end{exam}
 
  \begin{exam}
  The polynomial
 $$p(x) =\prod_{i=1}^N x_i,\quad x=(x_1, \cdots, x_N)\in \R^N$$
 is $a$-hyperbolic for any $a\in\R^N$ with $p(a)>0$. The $a$-eigenvalues of $x\in\R^N$ are given by
  $$\left\{\lambda_i(p; a, x), i=1,\cdots, N\right\}=\left\{x_i/a_i, i=1,\cdots, N\right\}.$$
  \end{exam}

Suppose $p$ is $a$-hyperbolic.
Observe from the definition of $a$-eigenvalues of $x$ that
 \begin{equation}
 \label{prodlam}
 \frac{p(x)}{p(a)}=\prod_{i=1}^k \lambda_i (p; a, x).
 \end{equation}
 Denote
 \begin{equation}
 \label{dpxa}
 p'_x(a) =\frac{d}{dt}\mid_{t=0} p(a+ tx).
 \end{equation}
 Then
 \begin{equation}
 \label{sumlam}
 \frac{p'_x(a)}{p(a)}= \sum_{i=1}^k \lambda_i (p; a, x). 
 \end{equation}
 Note that
 \begin{equation}
 \label{lameq}
 \lambda_i(p; a, a)=1;~\lambda_i(p; a, tx)= t\lambda_i(p; a, x) \text{ mod } \mathcal{S}_k,~\lambda_i(p; a, ta+ x)= t+ \lambda_i(p; a, x) \text{ mod } \mathcal{S}_k.
 \end{equation}
  If $p$ is be $a$-hyperbolic, then we denote its {\it edge at $a$} by
 $$E_a(p)=\{x\in\R^N: \lambda_1(p; a, x)=\cdots=\lambda_k(p; a, x)=0\}.$$
 We have 
 \begin{equation}
 \label{Eax}
 \lambda_i(p; a, x)=\mu  \quad \text{for all i }\Longleftrightarrow \lambda_i(p; a, x-\mu a)=0  \quad \text{for all i } \Longleftrightarrow x-\mu a\in E_a(p).
 \end{equation}
 The G\r{a}rding cone of $p$ at $a$ is defined to be
 $$\Gamma_a(p)=\{x\in \R^N: \lambda_i(p; a, x)>0\text{ for all } i=1,\cdots, k\}.$$
 A fundamental result of G\r{a}rding \cite[Theorem 2]{Garding} states that if $p$ is $a$-hyperbolic and $b\in \Gamma_a(p)$, then $p$ is $b$-hyperbolic and $\Gamma_b(p)= \Gamma_a(p)$. Therefore, we use $\Gamma(p)$ to denote $\Gamma_a(p)$ whenever $p$ is $a$-hyperbolic.  Another fundamental result of G\r{a}rding \cite[Theorem 3]{Garding}  says that the edge $E_a(p)$ of a hyperbolic polynomial $p$ at $a$ is equal to the {\it linearity $L(p)$ of $p$} where
 $$L(p)=\{x\in\R^N: p(tx + y)= p(y)\quad\text{for all } t\in \R \quad\text{and } y\in\R^N\}.$$ 
 For later reference, we summarize these results in the following theorem.
 \begin{thm}[G\r{a}rding] 
 \label{G59thm}
 Let $p$ be hyperbolic at $a\in\R^N$. Then
 \begin{myindentpar}{1cm}
 (i) If $b\in \Gamma_a(p)$, then $p$ is $b$-hyperbolic and $\Gamma_b(p)= \Gamma_a(p)$.\\
 (ii) $E_a(p)= L(p).$
 \end{myindentpar}
 \end{thm}
 From (\ref{prodlam}) and (\ref{sumlam}), 
 we obtain the following quantitative G\r{a}rding's inequality.
 \begin{lem}[Quantitative G\r{a}rding's inequality] 
 \label{Glem}
  Suppose $p$ is a homogenous real polynomial of degree $k$ on $\R^N$ and 
 $p$ is $a$-hyperbolic. If $x\in \Gamma(p)$, then
 \begin{equation}
 \label{QGar}
\frac{1}{k} \frac{p'_x(a)}{p(a)} \geq \left(\frac{p(x)}{p(a)}\right)^{1/k} + \frac{1}{k}\sum_{i=1}^k \left[\sqrt{ \lambda_i(p; a, x)}-  \left(\frac{p(x)}{p(a)}\right)^{\frac{1}{2k}} \right]^2. 
 \end{equation}
 In particular, if $k\geq 2$ and $x\in\Gamma(p)$ with 
 \begin{equation}
 \label{Eaxmu}
 \frac{1}{k} \frac{p'_x(a)}{p(a)} = \left(\frac{p(x)}{p(a)}\right)^{1/k}, 
 \end{equation}
 then there is a positive constant $\mu$ such that $x-\mu a\in E_a(p).$
 \end{lem}
  \begin{proof}
  Without the last nonnegative term, (\ref{QGar}) is the original G\r{a}rding's inequality whose proof uses (\ref{prodlam}), (\ref{sumlam}) and  the Cauchy inequality for $k$ positive numbers.

For the full version of (\ref{QGar}), we use (\ref{prodlam}), (\ref{sumlam})
and 
following quantitative version of Cauchy's inequality: If $x_1, \cdots, x_k$ are $k$ ($k\geq 2$) nonnegative numbers, 
then
\begin{multline}
\label{RCS}
\frac{1}{k}\sum_{i=1}^k x_i-(x_1\cdots x_k)^{\frac{1}{k}}-\frac{1}{k}\sum_{i=1}^k (\sqrt{x_i}-(x_1\cdots x_k)^{\frac{1}{2k}})^2 \\= 2 (x_1\cdots x_k)^{\frac{1}{2k}}\left(\frac{1}{k}\sum_{i=1}^k \sqrt{x_i}-(x_1\cdots x_k)^{\frac{1}{2k}}\right)\geq 0.
\end{multline}
Clearly, (\ref{QGar}) follows from (\ref{RCS}) applied to $\lambda_i(p; a,x)$. Moreover, if $k\geq 2$, and (\ref{Eaxmu}) holds, then we must have $\lambda_i(p; a, x)=\cdots= \lambda_k(p; a, x)=\mu$ for some positive constant $\mu$. Hence, the last assertion follows from (\ref{Eax}).
 \end{proof}
 \begin{exam}
 \label{Ex3}
 Let $N= \frac{1}{2}n(n+1)$ and let $A$ be a symmetric $n\times n$ matrix $A= (a_{ij})$. We can view $A$ as a point in $\R^N$. Then $P(A)=\det A$ is $A$-hyperbolic for any positive definite matrix $A$. Let $I_n$ be the identity $n\times n$ matrix. Define $P_k$ by
\begin{equation}
\label{PPk}
\det (tI_n + A)=P(t I_n + A)= \sum_{k=0}^n t^{n-k} P_k(A) \quad\text{for all } t\in\R.
\end{equation}
Then  $P_k$ is a homogenous polynomial of degree $k$ on $\R^N$; moreover, $P_k$ is $I_n$-hyperbolic (see, Example 3 and the discussion at the end of p. 959 in \cite{Garding}). 
\end{exam}

From now on, let $P_k$ be as in Example \ref{Ex3}. From this example, we know that $P_k$ is $I_n$-hyperbolic. Thus, for any symmetric $n\times n$ matrix $A$, we have from the definition of $I_n$-hyperbolicity that the $I_n$-eigenvalues $\lambda_i (P_k; I_n, A)$ are real numbers, for all $i=1,\cdots, k$.
\vglue 0.1cm
Suppose furthermore that $A$ is a symmetric $n\times n$ matrix with $\lambda(A)\in\Gamma_k$ (as defined in (\ref{Gak})). Then, from $\lambda_i (P_k; I_n, A)\in\R$, 
$$P_k(tI_n + A)=\sum_{i=0}^k {n-i\choose k-i}t^{k-i}\sigma_i(\lambda(A))= P_k(I_n)\prod_{i=1}^k (t+ \lambda_i (P_k; I_n, A))$$
and $\sigma_i(\lambda(A))>0$ for all $i$, we easily find that $\lambda_i (P_k; I_n, A)>0$ for all $i=1,\cdots, k$. Hence $A\in \Gamma(P_k)$ from which we deduce that $P_k$ is $A$-hyperbolic by Theorem \ref{G59thm}. Recall that we use $\Gamma(P_k)$ to denote the G\r{a}rding cone of $P_k$ at $I_n$. 
Vice versa, if $A\in \Gamma(P_k)$, then by definition, $\lambda_i (P_k; I_n, A)>0$ for all $i=1,\cdots, k$ and therefore, $\sigma_i(\lambda(A))>0$ for all $i=1,\cdots, k$ which show that $\lambda(A)\in \Gamma_k$.
Thus, we have
\begin{equation}
\label{GGk}
\Gamma(P_k)=\{A\in \R^N: \lambda(A)\in \Gamma_k\}.
\end{equation}
The following lemma shows the triviality of  the edge of $P_k$ when $k\geq 2$.
\begin{lem}
\label{E0lem}
 If $k\geq 2$, then 
 \begin{equation}
 \label{Eaxpk}
 E_{A_0}(P_k)=\{0\} \quad\text{ whenever }P_k\text{ is }A_0-\text{hyperbolic}.
 \end{equation}
 \end{lem}
 \begin{proof}
 In the proof, we use Theorem \ref{G59thm} $(ii)$ which implies that the edge $E_a(p)$ of a hyperbolic polynomial $p$ at $a$ does not depend on $a$.
 We apply this fact to $p=P_k$, and deduce that if $P_k$ is $A_0$-hyperbolic then $$E_{A_0}(P_k)=L(P_k)=E_{I_n}(P_k)=\{A\in\R^N: \lambda_1(P_k; I_n, A)=\cdots= \lambda_k(P_k; I_n, A)=0\}.$$ Let $A\in E_{I_n}(P_k)$. 
 Then $\lambda_1(P_k; I_n, A)=\cdots= \lambda_k(P_k; I_n, A)=0$ so the above expansion of $P_k(tI_n + A)$ shows that $\sigma_i(\lambda(A))=0$ for all $i=1,\cdots, k$. In particular, since $k\geq 2$, we find $$\sigma_1(\lambda(A))= \sigma_2(\lambda(A))=0.$$ Therefore, the eigenvalues $\lambda_1(A), \cdots, \lambda_n(A)$ of the symmetric matrix $A$ satisfy
 $$\sum_{i=1}^n [\lambda_i(A)]^2= [\sigma_1(\lambda(A))]^2-2\sigma_2(\lambda(A))=0.$$ It follows that $A$ is the $0$ matrix. This shows that $E_{A_0}(P_k)=E_{I_n}(P_k)=\{0\}$ as claimed.
 \end{proof}
 Note that the conclusion of Lemma \ref{E0lem} is false for $k=1$ since
 $$E_{A_0}(P_k)=L(P_1)=\{A\in \R^N: P_1(A)=\text{trace}(A)=0\}.$$
 We have the following lemma.
 \begin{lem}
 \label{lambm}
 Let $p$ be a homogenous real polynomial of degree $k$ on $\R^N$. Suppose that 
 $p$ is  $a$-hyperbolic with $E_a(p)=\{0\}$. 
 Assume that $\{b^{(m)}\}\subset\R^N$ satisfies $\lambda_i(p; a, b^{(m)})\rightarrow 0$ when $m\rightarrow \infty$ for all $i=1,\cdots, k$.  Then $b^{(m)}\rightarrow 0$ when $m\rightarrow \infty$. 
 \end{lem}
 The lemma is perhaps standard; however, we could not locate a precise reference so we include its proof here. In the proof, we use that $\lambda_i(p; a, x)$, modulo $\mathcal{S}_k$, is continuous in $x$ (see, \cite[p. 1105]{HLn}). This comes from the algebraic fact that roots of a degree k polynomial depend continuously on its coefficients.
 \begin{proof}[Proof of Lemma \ref{lambm}]
 We first show that $b^{(m)}$ is bounded. Suppose that $\|b^{(m)}\|= M_m\rightarrow \infty$. Consider $\tilde b^{(m)} = \frac{b^{(m)}}{M_n}$. Then $\|\tilde b^{(m)}\|=1$ while, modulo $\mathcal{S}_k$, 
$$\lambda_i(p; a, \tilde b^{(m)})= \frac{\lambda_i(p; a, b^{(m)})}{M_n}\rightarrow 0\quad \text{for all } i=1,\cdots, k. $$
Up to extracting a subsequence, we have $\tilde b^{(m)}\rightarrow b$ with $\|b\|=1$ while $\lambda_i(p; a, \tilde b^{(m)})\rightarrow \lambda_i(p; a, b)= 0$ for all $i=1,\cdots, k$. Thus, $b\in E_a(p)$ which shows that $b=0$, a contradiction. 

Next, we show that $b^{(m)}$ converges to $0$. We already known that there is $M>0$ such that $\|b^{(m)}\|\leq M$ for all $m$. Suppose there exists $\delta>0$ such that, there is a subsequence, still denoted $b^{(m)}$, satisfying $M\geq \|b^{(m)}\|\geq \delta>0$. We use compactness as above to get a 
$b$ with $\|b\|=1$ while $\lambda_i(p; a, b)= 0$ for all $i=1,\cdots, k$, a contradiction. 
\end{proof}
\section{Nonlinear integration by parts inequalities}
\label{ibpksect}
In this section, we prove Proposition \ref{ibplem} which is concerned with $P_k$ and its extensions to other hyperbolic polynomials.
\begin{proof}[Proof of Proposition \ref{ibplem}]   Since $u, v\in C^{1,1}(\overline{\Omega})\cap C^3(\Omega)$ are $k$-admissible functions with $u=v=0$ on $\p\Omega$, we have $u, v\leq 0$ in $\Omega$. We view $S_k$ as a function on $n\times n$ matrices $r=(r_{ij})_{1\leq i, j\leq n}$ where
$$S_k(r) =\sigma_k (\lambda(r)).$$
Let
$$S_k^{ij}(D^2 u)=\frac{\p}{\p r_{ij}} S_k(D^2 u).$$
Then, it is well-known that (see, for example, \cite{Re, W1, W2})
$$S_k(D^2 u)=\frac{1}{k} \sum_{i, j=1}^n S_k^{ij}(D^2 u) D_{ij} u$$
and, for each $i=1,\cdots, n$, we have the following divergence-free property of the matrix $(S_k^{ij}(D^2 u))$:
$$\sum_{j=1}^n D_j S_k^{ij}(D^2 u)=0.$$
Therefore, integrating by parts twice, we get
\begin{eqnarray}
\label{ibpuv}
\int_\Omega |v| S_k(D^2 u)dx &=&  \frac{1}{k}  \int_{\Omega}\sum_{i, j=1}^n (-v)  S_k^{ij}(D^2 u) D_{ij} u~dx\nonumber\\&=&
 \frac{1}{k}  \int_{\Omega} \sum_{i, j=1}^n  D_j[v S_k^{ij}(D^2 u)] D_{i} u~dx= \frac{1}{k}  \int_{\Omega} \sum_{i, j=1}^n  D_j v S_k^{ij}(D^2 u) D_{i} u~dx\nonumber\\&=&
 \frac{1}{k}  \int_\Omega (-u) \sum_{i, j=1}^n S_k^{ij}(D^2 u) D_{ij} v dx=  \frac{1}{k} \int_\Omega |u| \sum_{i, j=1}^n S_k^{ij}(D^2 u) D_{ij}v dx.
\end{eqnarray}
We need to show that
\begin{equation}
\label{keyG}
 \frac{1}{k}  \sum_{i, j=1}^n S_k^{ij}(D^2 u) D_{ij} v\geq  [S_k(D^2 u)]^{\frac{k-1}{k}} [S_k(D^2 v)]^{\frac{1}{k}}.
\end{equation}
Let $P_k$ be as in (\ref{PPk}). We use the notation $p'_x (a)$ as defined by (\ref{dpxa}).
 Note that, for $C^2$ functions $u$ and $v$, we have
$$S_k(D^2 u)= P_k (D^2 u),\quad \text{and } (P_k)^{'}_{D^2 v}(D^2 u)= \sum_{i, j=1}^n S_k^{ij}(D^2 u) D_{ij} v.$$ 
Since $u$ and $v$ are $k$-admissible, we have
$$D^2 u, D^2 v\in \Gamma(P_k).$$ 
Thus, by G\r{a}rding's inequality (Lemma \ref{Glem}), 
$$ \frac{1}{k} \sum_{i, j=1}^n S_k^{ij}(D^2 u) D_{ij} v =  \frac{1}{k}(P_k)^{'}_{D^2 v}(D^2 u) \geq P_k(D^2 u)\left( \frac{P_k(D^2 v)}{P_k(D^2 u)}\right)^{1/k} = [P_k(D^2 u)]^{\frac{k-1}{k}} [P_k(D^2 v)]^{\frac{1}{k}}.$$
Therefore, (\ref{keyG}) holds and we obtain (\ref{NIBP}).

If $k\geq 2$ and  the equality holds in (\ref{NIBP}), then (\ref{keyG}) must be an equality for almost all $x\in\Omega$. For those $x$, using the last assertion of Lemma \ref{Glem}, we can find a positive number $\mu(x)$ such that
$$D^2 u(x)-\mu(x) D^2 v(x) \in E_{D^2 u(x)}(P_k)=\{0\}$$
where we used (\ref{Eaxpk}) in the last equality. Since $u, v\in C^3(\Omega)$, $\mu$ is a continuous function on $\Omega$
and $$D^2 u(x)= \mu(x) D^2 v(x) \quad \text{ for all }x\in \Omega.$$
The proof of the proposition is complete.
\end{proof}
A particular consequence of Proposition \ref{ibplem} is the following corollary. 
\begin{cor}
 \label{ibpcor}
  Let $\Omega$ be a bounded, open, smooth, uniformly $(k-1)$-convex (if $k\geq 2$) domain in $\R^n$. Let  $w\in C^{1,1}(\overline{\Omega})\cap C^{\infty}(\Omega)$ be a $k$-Hessian eigenfunction as in (\ref{kEVP_eq}).
  Then for any $k$-admissible function $v\in C^{1,1}(\overline{\Omega})\cap C^3(\Omega)$ with $v=0$ on $\p\Omega$, one has
\begin{equation}
\label{NIBPcor}
\lambda(k;\Omega)\int_\Omega |v| |w|^kdx \geq \int_\Omega |w|^k [S_k(D^2 v)]^{\frac{1}{k}} dx.
\end{equation}
\end{cor}
Corollary \ref{ibpcor} is sharp since equality holds when $v$ is  a $k$-Hessian eigenfunction of $\Omega$. When $k=n$, (\ref{NIBPcor}) can be viewed as a reverse version of 
the celebrated Aleksandrov's maximum principle for the Monge-Amp\`ere equation (see \cite[Theorem 2.8]{F2} and \cite[Theorem 1.4.2]{G01}) which states:
 If $u\in C(\overline{\Omega})$ is a convex function on an open, bounded and convex domain  $\Omega\subset\R^n$ with
 $u=0$ on $\p \Omega$, then
\begin{equation}
\label{Alek_est}
|u(x)|^{n}\le C(n)(\text{diam }\Omega)^{n-1}\text{dist }(x,\partial \Omega)\int_{\Omega}\det D^2 u~dx\qquad \text{ for all } x\in\Omega.
\end{equation}

In fact, the reverse Aleksandrov estimate holds for more relaxed conditions on the domains and convex functions involved.

\begin{prop}[Reverse Aleksandrov estimate]
\label{ReA}
 Let $\Omega$ be a bounded open convex domain in $\R^n$.  Let $\lambda[n; \Omega]$ be the Monge-Amp\`ere eigenvalue of $\Omega$ and let $w$ be a nonzero Monge-Amp\`ere eigenfunction of $\Omega$ (see also (\ref{EVP_eq})).
Assume that $u\in C^{5}(\Omega)\cap C(\overline{\Omega})$ is a strictly convex function in $\Omega$ with  $u=0$ on $\Omega$ and satisfies
$$\int_\Omega (\det D^2 u)^{1/n} |w|^{n-1}~dx<\infty.$$
Then
\begin{equation}
\label{ReA2}
\lambda[n; \Omega]  \int_\Omega |u| |w|^n~dx\geq \int_\Omega (\det D^2 u)^{1/n} |w|^n~dx.
\end{equation}
\end{prop}
\begin{proof}[Proof of Proposition \ref{ReA}]
For the proof, we recall the {\it nonlinear integration by parts} inequality established in \cite[Proposition 1.7]{LSNS} (see also \cite{Le21b}). It says that if $u, v\in C(\overline{\Omega})\cap C^5 (\Omega)$ are strictly convex functions in $\Omega$ with $u=v=0$ on $\p\Omega$ and if 
 \begin{equation*}\int_{\Omega}(\det D^2 u)^{\frac{1}{n}}  (\det D^2 v)^{\frac{n-1}{n}}~dx<\infty,~\text{and}~\int_{\Omega}\det D^2 v~dx<\infty,
 \end{equation*} then
\begin{equation} 
\label{IBPn}
\int_{\Omega} |u|\det D^2 v~dx \geq \int_{\Omega} |v|(\det D^2 u)^{\frac{1}{n}} (\det D^2 v)^{\frac{n-1}{n}}~dx.
\end{equation}
We apply (\ref{IBPn}) to $u$ and $v=w$. 
Then, using $\det D^2 w= (\lambda[n;\Omega] |w|)^n$, we get
\begin{eqnarray*}(\lambda[n; \Omega])^n \int_\Omega  |u| |w|^n= \int_\Omega  |u| \det D^2 w~dx&\geq& \int_\Omega |w|(\det D^2 u)^{1/n} (\det D^2 w)^{\frac{n-1}{n}}~dx
\\&=&  ( \lambda[n;\Omega])^{n-1} \int_\Omega (\det D^2 u)^{1/n} |w|^n~dx.
\end{eqnarray*}
Dividing the first and last expressions in the above estimates by $ (\lambda[n;\Omega])^{n-1}$, we obtain (\ref{ReA2}).
\end{proof}

\begin{rem}
\label{ext_rem}
The method of proof of Proposition \ref{ibplem} relies on the divergence form structure of the $k$-Hessian operator $S_k(D^2 u)$.
If we replace $P_k(A)$ in the proof of  Proposition \ref{ibplem} by other homogeneous, hyperbolic polynomials $P(A)$ of degree $K$, then the conclusion still holds as long as the following conditions are satisfied:
\begin{enumerate}
\item[(P1)] Let
$$P^{ij}(D^2 u)=\frac{\p}{\p r_{ij}} P(D^2 u).$$
Then
$$P(D^2 u)= \frac{1}{K}\sum_{i, j=1}^n P^{ij} (D^2 u) D_{ij} u.$$
\item[(P2)]  For each $u\in C^3(\Omega)$ and $i=1,\cdots, n$, we have the following divergence-free property of the matrix $(P^{ij}(D^2 u))$:
$$\sum_{j=1}^n D_j P^{ij}(D^2 u)=0.$$
\end{enumerate}
Due to the homogeneity of $P$, property $(P1)$ always holds, in view of Euler's formula.  
The properties $(P1)$ and $(P2)$ hold for the following hyperbolic polynomials
$$[P_k(A)]^l\quad\text{where } l=1, 2,\cdots.$$
Note that
$$K= kl,\quad \text{and }\Gamma(P_k)= \Gamma ([P_k]^l),$$
so we obtain the following result stated in Proposition \ref{ibpleml}. 
\end{rem}
 
 \begin{prop}
 \label{ibpleml}
  Let $\Omega$ be a bounded, open, smooth domain in $\R^n$. Assume that $\p\Omega$ is uniformly $(k-1)$-convex if $k\geq 2$. Let $l$ be a positive integer. Then,
 for $k$-admissible functions $u, v\in C^{1,1}(\overline{\Omega})\cap C^3(\Omega)$ with $u=v=0$ on $\p\Omega$, one has
\begin{equation}
\label{NIBPl}
\int_\Omega |v| [S_k(D^2 u)]^l~dx \geq \int_\Omega |u| [S_k(D^2 u)]^{\frac{kl-1}{k}} [S_k(D^2 v)]^{\frac{1}{k}} dx.
\end{equation}
If $k\geq 2$ and  the equality holds in (\ref{NIBPl}), then there is a positive, continuous function $\mu$ such that $$D^2 u(x)= \mu(x) D^2 v(x) \quad \text{ for all }x\in \Omega.$$
\end{prop}
 \begin{rem} If $k\geq 2$, then the quantity
 $\int_\Omega |v| S_k(D^2 u)dx$ in Proposition \ref{ibplem} is called the non-commutative inner product of two functions $v$ and $u$ on the cone of $k$-admissible functions
 in Verbitsky \cite{V}.  Verbitsky proved in \cite[Theorem 3.1]{V} the following fully nonlinear Schwarz's inequality
 \begin{equation}
 \label{vbineq}
 \int_\Omega |v| S_k(D^2 u)dx \leq  \left(\int_\Omega |u| S_k(D^2 u)dx\right)^{\frac{k}{k+1 }}  \left(\int_\Omega |v| S_k(D^2 v)dx\right)^{\frac{1}{k+1 }} 
 \end{equation}
 which has many applications in the Hessian Sobolev inequalities.
 
 We also note that the proof of (\ref{vbineq}) in \cite{V} also used exactly the properties of $P$ in Remark \ref{ext_rem}. Thus, for $k$-admissible functions $u, v\in C^{1,1}(\overline{\Omega})\cap C^3(\Omega)$ with $u=v=0$ on $\p\Omega$, we also have
 \begin{equation}
 \label{vbineql}
 \int_\Omega |v| [S_k(D^2 u)]^l dx \leq  \left(\int_\Omega |u| [S_k(D^2 u)]^l dx\right)^{\frac{kl}{kl+1 }}  \left(\int_\Omega |v| [S_k(D^2 v)]^l dx\right)^{\frac{1}{kl+1 }}. 
 \end{equation}
 \end{rem}

 To conclude this section, we note that there are homogeneous hyperbolic polynomials $P$ which do not have property (P2) in Remark \ref{ext_rem}. We may call these {\it non-divergence form} hyperbolic polynomials. One
 example is following Monge-Amp\`ere type operator
\begin{equation}
\label{MpMA}
\mathcal{M}_{n-1}(D^2 u):= \det\left ((\Delta u)I_n-D^2 u\right)
\end{equation}
which appears in many geometric contexts, both real and complex; see, for example \cite{HLn2, Sha, ToW} and the references therein. When $n=3$, we have \begin{eqnarray*}
P(D^2 u):=\mathcal{M}_2 (D^2 u)=  \det\left ((\Delta u)I_3-D^2 u\right)=S_1(D^2 u) S_2(D^2 u) - S_3(D^2 u).
\end{eqnarray*}
For $u(x)= x_1^3 + x_2^2 + x_3^2,$ one can check, using the divergence-free property of the matrices $S_k^{ij}$ for $k=1, 2, 3$, that
$$\sum_{j=1}^3 D_j P^{1j}(D^2 u)=\sum_{j=1}^3 (S^{1j}_1(D^2 u)  \frac{\p}{\p x_j} S_2(D^2 u) +  \frac{\p}{\p x_j} S_1(D^2 u) S^{1j}_2(D^2 u))=48\neq 0.$$

\section{A spectral characterization of the $k$-Hessian eigenvalue via dual G\r{a}rding cone}
\label{dual_sect}
In this section, we prove Theorem \ref{kHessL}. 

Let $\Gamma_k$ and $\Gamma_k^\ast$ be as in (\ref{Gak}) and (\ref{Gaka}), respectively.
We recall the following result of Kuo-Trudinger \cite[Proposition 2.1]{KT}.
\begin{prop}
\label{KTprop}
For matrices $B=(b_{ij})\in \Gamma_k$, $A=(a_{ij})\in \Gamma_k^\ast$, $k=1,\cdots, n$, we have 
$$[S_k(B)]^{1/k} \rho_k^\ast(A) \leq \frac{1}{n} {n\choose k}^{1/k} \trace (AB).$$
\end{prop}
\begin{proof}[Proof of Theorem \ref{kHessL}]
Let $w\in C^{\infty}(\Omega)\cap C^{1,1}(\overline{\Omega})$ be a nonzero $k$-Hessian eigenfunction so $w$ satisfies (\ref{kEVP_eq}). Then $D^2 w\in\Gamma_k$. 
Let $A= (a_{ij})\in V_k$. Then 
$$\rho_k^\ast(A)\geq \frac{1}{n}{n \choose k}^{1/k}.$$
Applying Proposition \ref{KTprop} to $D^2 w$ and $A$, we have
\begin{equation}
\label{Skop}
[S_k(D^2 w)]^{1/k} \leq \frac{1}{\rho_k^\ast(A) }\frac{1}{n} {n\choose k}^{1/k} \trace (AD^2 w)\leq a_{ij} D_{ij} w.
\end{equation}
Since $[S_k(D^2 w)]^{1/k} =\lambda(k; \Omega)|w|=-\lambda(k;\Omega) w$, we obtain
$$a_{ij} D_{ij} w+ \lambda (k;\Omega) w\geq 0\quad\text{in } \Omega.$$
By \cite[Proposition A.2(ii)]{Lions}, we find that
$$\lambda (k;\Omega)\leq \lambda_1^A.$$
Hence
$$\lambda(k;\Omega) \leq \inf_{A\in V_k} \lambda^A_1.$$
Now, we show that the infimum is achieved.  Note that, if $u$ is $k$-admissible, then $(S^{ij}_k(D^2 u))_{i\leq i, j\leq n} \in\Gamma_k^{\ast}$. Moreover, as a consequence of G\r{a}rding's inequality (\ref{keyG}), we find
 $$\rho_k^\ast (S^{ij}_k(D^2 u))= \frac{k}{n} [S_k(D^2 u)]^{\frac{k-1}{k}}{n \choose k}^{1/k}.$$
Observe that
$$-\lambda(k;\Omega)w= [S_k(D^2 w)]^{\frac{1}{k}} =  [S_k(D^2 w)]^{\frac{-(k-1)}{k}}  S_k(D^2 w)=\frac{1}{k}[S_k(D^2 w)]^{\frac{-(k-1)}{k}} S_k^{ij} (D^2 w) D_{ij}w. $$
Thus $\lambda(k;\Omega)$ is the first eigenvalue of 
$-a_{ij} D_{ij}$ where
$$(a_{ij})_{1\leq i, j\leq j}=\left (\frac{1}{k}[S_k(D^2 w)]^{\frac{-(k-1)}{k}} S_k^{ij} (D^2 w)\right)_{1\leq i, j\leq n}\in V_k\quad\text{with } \rho_k^\ast((a_{ij})) = \frac{1}{n}{n \choose k}^{1/k}.$$
\end{proof}
\begin{rem}
Let $V_k$ be as in (\ref{Vkeq}).
Observe from (\ref{Skop}) that for $u$ k-admissible, we have
$$[S_k(D^2 u)]^{1/k}=\inf_{A=(a_{ij})\in V_k} a_{ij} D_{ij} u.$$
\end{rem}
\section{Convergence to the $k$-Hessian eigenvalue}
\label{conRksect}
In this section, we prove Theorem \ref{conRk}.
 \begin{proof}[Proof of Theorem \ref{conRk}]
 $(i)$ For $m\geq 0$, 
 multiplying both sides of (\ref{kIS}) by $|u_{m+1}|$ and then integrating over $\Omega$, we find
 \begin{eqnarray*}
 R_k(u_{m+1}) \|u_{m+1}\|^{k+1}_{L^{k+1}(\Omega)}&=& \int_\Omega |u_{m+1}| S_k(D^2 u_{m+1}) dx\\&=& R_k(u_m) \int_{\Omega} |u_m|^k |u_{m+1}| dx + \frac{1}{(m+1)^2} \int_\Omega |u_{m+1}| dx
 \\&\leq& R_k(u_{m})\|u_m\|^k_{L^{k+1}(\Omega)} \|u_{m+1}\|_{L^{k+1}(\Omega)} + \frac{|\Omega|^{\frac{k}{k+1}}}{(m+1)^2} \|u_{m+1}\|_{L^{k+1}(\Omega)}.
 \end{eqnarray*}
It follows that
\begin{equation}
\label{monok}
 R_k(u_{m+1}) \|u_{m+1}\|^{k}_{L^{k+1}(\Omega)}\leq R_k(u_{m})\|u_m\|^k_{L^{k+1}(\Omega)} + \frac{ |\Omega|^{\frac{k}{k+1}}}{(m+1)^2}.
\end{equation}
Therefore, by iterating, we obtain
\begin{eqnarray*}
 R_k(u_{m+1}) \|u_{m+1}\|^{k}_{L^{k+1}(\Omega)}&\leq& R_k(u_{0})\|u_0\|^k_{L^{k+1}(\Omega)} + |\Omega|^{\frac{k}{k+1}}\sum_{m=0}^{\infty}  \frac{1}{(m+1)^2}\\&=& R_k(u_{0})\|u_0\|^k_{L^{k+1}(\Omega)} +\frac{\pi^2}{6} |\Omega|^{\frac{k}{k+1}}.
\end{eqnarray*}
 From (\ref{klamR}), we know that
 \begin{equation}
 \label{Rklow}
 R_k^{1/k}(u_m)\geq \lambda(k;\Omega)\quad \text{for } m\geq 1.
 \end{equation}
Hence, there exists a constant  $C_1(k, u_0, \Omega)$ independent of $m$ such that
\begin{equation} 
\label{umC}
\|u_{m+1}\|_{L^{k+1}(\Omega)}\leq C_1(k, u_0,\Omega).
\end{equation}

By the uniqueness (up to positive multiplicative constants) of the $k$-Hessian eigenfunctions, we can assume that $w\in C^{\infty}(\Omega)\cap C^{1,1}(\overline{\Omega})$ is a $k$-Hessian eigenfunction with $L^{\infty}$ norm $1$, that is
\begin{equation}
\label{kEVPw}
S_k (D^2 w)~=[\lambda(k;\Omega)]^k |w|^{k} \h~\text{in} ~\Omega,~
w =0\h~\text{on}~\p \Omega, \quad\text{and }\|w\|_{L^{\infty}(\Omega)}=1.
\end{equation}
Then, we use the nonlinear integration by parts inequality (\ref{NIBP}) to get

\begin{equation}
\label{um1w}
\int_\Omega |u_{m+1}| S_k(D^2 w)dx \geq \int_\Omega |w| [S_k(D^2 w)]^{\frac{k-1}{k}} [S_k(D^2 u_{m+1})]^{\frac{1}{k}} dx.
\end{equation}
Therefore, recalling (\ref{kEVPw}), we find after dividing both sides of the above inequality by $[\lambda(k;\Omega)]^k$ that
\begin{eqnarray}
\label{upum1}
\int_\Omega |u_{m+1}| |w|^k dx &\geq&  \int_\Omega [\lambda(k;\Omega)]^{-1} |w|^k \left[ R_k(u_m) |u_m|^k + \frac{1}{(m+1)^2}\right]^{\frac{1}{k}} dx\nonumber\\&>& \int_\Omega [\lambda(k;\Omega)]^{-1} |w|^k [R_k(u_m)]^{\frac{1}{k}} |u_m|dx\nonumber\\
&=&
 \int_{\Omega}  |u_m| |w|^k~dx + [R_k^{1/k}(u_m)-\lambda(k;\Omega)] [\lambda(k;\Omega)]^{-1} \int_{\Omega} |u_m| |w|^k~dx.
 \end{eqnarray}
Thus, (\ref{upum1}) together with (\ref{Rklow}) implies that the sequence $\left\{\int_{\Omega}|u_m| |w|^k~dx\right\}_{m=1}^{\infty}$ is increasing. On the other hand, using (\ref{umC}) and (\ref{kEVPw}), we find that
 \begin{equation*}
 \int_{\Omega}|u_m| |w|^k~dx \leq \int_{\Omega} |u_m| ~dx\leq C_2(k, u_0,\Omega).
 \end{equation*}
 It follows from $u_1<0$ in $\Omega$ that 
 $\int_{\Omega}|u_m| |w|^k~dx$  converges to a limit 
\begin{equation}
\label{Llim}
\lim_{m\rightarrow \infty} \int_\Omega |u_m||w|^k ~dx=L\in (0, \infty).
\end{equation}
For $m\geq 1$, we get from (\ref{upum1}) that
\begin{eqnarray}
\label{Rum}
 R_k^{1/k}(u_m)-\lambda(k;\Omega) &\leq& \lambda(k;\Omega)\frac{  \int_{\Omega}(|u_{m+1}|- |u_{m}|) |w|^k~dx} {\int_{\Omega} |u_m| |w|^k~dx}
\nonumber\\ &\leq&  \lambda(k;\Omega)\frac{  \int_{\Omega}(|u_{m+1}|- |u_{m}|) |w|^k~dx} {\int_{\Omega} |u_1| |w|^k~dx}.
\end{eqnarray}
Letting $m\rightarrow\infty$  in (\ref{Rum}) and recalling (\ref{Llim}), we conclude that the whole sequence $R_k(u_m)$ converges to $[\lambda(k;\Omega)]^k$ as asserted in (\ref{Rku_con}). 

In particular, we have
$
R_k(u_m)\leq C_3(k, u_0,\Omega)
$
and hence, using the H\"older inequality and (\ref{umC}), 
\begin{equation}
\label{SkL1}
\int_{\Omega} S_k(D^2 u_{m+1})~dx= (m+1)^{-2}|\Omega| + R_k(u_m)\int_{\Omega}|u_m|^k dx \leq  C_4(k, u_0,\Omega).
\end{equation}
$(ii)$ From (\ref{umC})
 and the local $W^{1, q}_{loc}(\Omega)$ estimate for the $k$-Hessian equation (see, Theorem \ref{TWthm} below) for all $q<\frac{nk}{n-k}$, we have the uniform bound for $u_m$ in $W^{1,q}(V)$ for each $V\subset\subset \Omega$. 
 Thus, there exists a subsequence $(u_{m_j})$ that converges weakly in $W^{1, q}_{loc}(\Omega)$ for all $q<\frac{nk}{n-k}$ to a function $u_{\infty}\in W^{1, q}_{loc}(\Omega)$. From the compactness of the Sobolev embedding $W^{1, q}$ to $L^k$ on smooth bounded sets for all $q$ sufficiently close to $\frac{nk}{n-k}$, we can also assume that $u_{m_j}$ converges strongly to $u_{\infty}$ in $L^k_{loc}(\Omega)$.  
 From the second inequality in (\ref{SkL1}) and Fatou's lemma, we have
 \begin{equation}
 \label{umLk1}
 \infty>\liminf_{j\rightarrow\infty} \int_\Omega |u_{m_j}|^k~dx\geq \int_\Omega |u_{\infty}|^k~dx.
 \end{equation}
 On the other hand, using (\ref{umC}), we find that for each $V\subset\subset \Omega$,
 \begin{eqnarray*} \int_\Omega |u_{m_j}|^k~dx =  \int_{\Omega\setminus V} |u_{m_j}|^k~dx +  \int_V |u_{m_j}|^k~dx&\leq& \|u_{m_j}\|^k_{L^{k+1}(\Omega\setminus V)}|\Omega\setminus V|^{\frac{1}{k+1}} + \int_V |u_{m_j}|^k~dx
 \\&\leq& C_1^{k}|\Omega\setminus V|^{\frac{1}{k+1}} + \int_V |u_{m_j}|^k~dx.
 \end{eqnarray*}
 Therefore, using the strong convergence of $u_{m_j}$ to $u_{\infty}$ in $L^k(V)$, we get
  \begin{equation}
 \label{umLk2}\limsup_{j\rightarrow\infty} \int_\Omega |u_{m_j}|^k~dx \leq C_1^{k}|\Omega\setminus V|^{\frac{1}{k+1}} + \int_V |u_{\infty}|^k~dx \leq C_1^{k}|\Omega\setminus V|^{\frac{1}{k+1}} + \int_\Omega |u_{\infty}|^k~dx .
 \end{equation}
 Combining (\ref{umLk1}) with (\ref{umLk2}), we obtain (\ref{umLk}) as claimed. Clearly, (\ref{umLk}) and the increasing property of $\int_{\Omega}|u_m| |w|^k~dx$ implies that $u_{\infty}$ is nonzero.
 
 Finally, from (\ref{Rum}), (\ref{umLk}) and  the increasing property of $\{\int_{\Omega}|u_m| |w|^k~dx\}_{m=1}^{\infty}$, we obtain (\ref{umRate}).
 \\
 $(iii)$ Assume now $k=n$. We show the convergence of $u_m$ to a nontrivial Monge-Amp\`ere eigenfunction $u_{\infty}$ of $\Omega$. Similar result was proved in \cite{AK}. However, our scheme (\ref{kIS}) and approach are a bit different, so we include the details. 
 
 As mentioned in the introduction, we can define the Rayleigh quotient $R_n(u)$ (for the Monge-Amp\`ere operator), as in (\ref{RQk}), of a nonzero merely convex function $u$
 where
$\det D^2 u~ dx$ is interpreted as the Monge-Amp\`ere measure $Mu$ associated with $u$. It is defined by
$$Mu(E) = |\p u(E)|~\text{where } \p u(E) = \bigcup_{x\in E} \p u(x),~\text{for each Borel set } E\subset\Omega$$
where
$$
\partial u (x):=\{p\in \R^{n}\,:\, u(y)\ge u(x)+p\cdot (y-x)\quad \forall\, y \in \Omega\}.
$$
In what follows, when $u$ is merely convex, $R_n (u)$ and $\det D^2 u$ are understood in the above sense.

Applying the Aleksandrov estimate (\ref{Alek_est}) to $u_{m+1}$ where $m\geq 0$, and invoking (\ref{SkL1}), we find
\begin{eqnarray*}
\|u_{m+1}\|^n_{L^{\infty}(\Omega)} \leq C(n,\Omega)\int_{\Omega}\det D^2 u_{m+1}~ dx\leq  C(n,\Omega) C_4(n, u_0,\Omega)\leq C(n,\Omega, u_0).
\end{eqnarray*} 
Hence, we obtain the uniform $L^{\infty}$ bound
$$\|u_m\|_{L^{\infty}(\Omega)} \leq C(n,\Omega, u_0)<\infty.$$
Again, the Aleksandrov estimate and the convexity of $u_m$ give the uniform $C^{0, \frac{1}{n}} (\overline{\Omega})$ bound for $u_m$:
$$\|u_m\|_{C^{0, \frac{1}{n}} (\overline{\Omega})} \leq C(n,\Omega, u_0)\quad\text{for all } m\geq 1.$$ 
Therefore, up to extracting a subsequence, we have the following uniform convergence $$u_{m_j}\rightarrow u_{\infty}\not \equiv 0$$
for a convex function $u_{\infty}\in C(\overline{\Omega})$ with $u_{\infty}=0$ on $\p\Omega$ while we also have the uniform convergence 
 $$u_{m_j + 1}\rightarrow w_{\infty}\not\equiv 0$$ 
 for a convex function $w_{\infty}\in C(\overline{\Omega})$ with $w_{\infty}=0$ on $\p\Omega$.
 
Thus, letting $j\rightarrow\infty $ in $$\det D^2 u_{m_{j} +1}=R_n(u_{m_j})|u_{m_j}|^n + (m_j+1)^{-2},$$ using (\ref{Rku_con}) and the weak convergence of the Monge-Amp\`ere measure (see \cite[Corollary 2.12]{F2} and \cite[Lemma 5.3.1]{G01}), we get
\begin{equation}
\label{uwinfi}
\det D^2 w_{\infty}= (\lambda(n;\Omega) |u_{\infty}|)^n.
\end{equation}
In view of (\ref{monok}), we have
\begin{equation}
\label{monok2}
 R_n(u_{m_j+1}) \|u_{m_j+1}\|^{n}_{L^{n+1}(\Omega)}\leq R_n(u_{m_j})\|u_{m_j}\|^n_{L^{n+1}(\Omega)} +  |\Omega|^{\frac{n}{n+1}} (m_j+1)^{-2}.
\end{equation}
Letting $j\rightarrow \infty$ in (\ref{monok2}) and recalling (\ref{Rku_con}), we first find that
$$\|w_{\infty}\|_{L^{n+1}(\Omega)} \leq \|u_{\infty}\|_{L^{n+1}(\Omega)}. $$
In fact, we have the equality. To see this, we use $m_{j+2}\geq m_j + 2$ and iterate (\ref{monok}) from $m_{j}+ 1$ to $m_{j+2}-1$ to get
$$ R_n(u_{m_{j+2}}) \|u_{m_{j+2}}\|^{n}_{L^{n+1}(\Omega)}\leq R_n(u_{m_j+1})\|u_{m_j+1}\|^n_{L^{n+1}(\Omega)} +  |\Omega|^{\frac{n}{n+1}} \sum_{s= m_j + 2}^{m_{j+2}}s^{-2}.$$
Again,  letting $j\rightarrow \infty$ in the above inequality and recalling (\ref{Rku_con}), we obtain 
$$\|u_{\infty}\|_{L^{n+1}(\Omega)} \leq \|w_{\infty}\|_{L^{n+1}(\Omega)}. $$
In conclusion, we have
$$\|w_{\infty}\|_{L^{n+1}(\Omega)} = \|u_{\infty}\|_{L^{n+1}(\Omega)}. $$
However, from (\ref{uwinfi}), we have
\begin{eqnarray*}R_n(w_\infty) \|w_{\infty}\|^{n+1}_{L^{n+1}(\Omega)}=\int_{\Omega} |w_{\infty}| \det D^2 w_{\infty}~dx &=&[\lambda(n; \Omega)]^n \int_{\Omega} |u_{\infty}|^n |w_{\infty}| ~dx\\ &\leq& 
[\lambda(n; \Omega)]^n \|u_{\infty}\|^n_{L^{n+1}(\Omega)}  \|w_{\infty}\|_{L^{n+1}(\Omega)}\\&=& [\lambda(n; \Omega)]^n\|w_{\infty}\|^{n+1}_{L^{n+1}(\Omega)}.
\end{eqnarray*}
Since, by (\ref{lam_def}), $R_n(w_{\infty})\geq  (\lambda[n; \Omega])^n=[\lambda(n; \Omega)]^n$,  we must have $R_n(w_{\infty})= [\lambda(n; \Omega)]^n$, and the inequality above must be an equality,  
but this gives $u_{\infty}= c w_{\infty}$ for some constant $c>0$. Thus, from (\ref{uwinfi}), we have 
\begin{equation}
\label{winftyEV}
\det D^2 w_{\infty} = c^n [\lambda(n; \Omega)]^n |w_\infty|^n.
\end{equation}
Note that the quantities $\lambda (n;\Omega)$ in (\ref{lam_def1}) and $\lambda[n;\Omega]$ in (\ref{lam_def}) 
are a priori different. In \cite{LSNS}, the bracket notation $\lambda[n;\Omega]$ is most relevant for $\Omega$ with corners or flat parts on $\p\Omega$. However, when $\Omega$ is a smooth, bounded and uniformly convex domain, it was shown in \cite{LSNS} that $\lambda(n;\Omega)= \lambda [n;\Omega].$

By the uniqueness of the Monge-Amp\`ere eigenfunctions (\cite[Theorem 1]{Lions} and \cite[Theorem 1.1]{LSNS}), it follows from (\ref{winftyEV}) that $c=1$ and  $w_\infty= u_{\infty}$ is a Monge-Amp\`ere eigenfunction of $\Omega$.
From (\ref{Llim}), we have
$$\int_\Omega |u_{\infty}| |w|^n ~dx=\lim_{m\rightarrow \infty} \int_\Omega |u_m||w|^n~dx=L.$$
With this property and the uniqueness up to positive multiplicative constants of the Monge-Amp\`ere eigenfunctions of $\Omega$, we conclude that the limit $u_{\infty}$ does not depend on the subsequence $u_{m_j}$. This shows that the whole sequence $u_m$ converges to a nonzero Monge-Amp\`ere eigenfunction $u_{\infty}$ of $\Omega$. \\
$(iv)$ When $k=1$, we prove the full convergence in $W^{1,2}_0(\Omega)$ of $u_m$ to a first Laplace eigenfunction $u_\infty$ of $\Omega$. We sketch its proof along the lines of $(iii)$. Note that the Rayleigh quotient for $R_1(u)$ in (\ref{RQk}) is defined originally for $u\in C^2(\Omega)$. If furthermore, $u\leq 0$ in $\Omega$ and $u=0$ on $\p\Omega$, then an integration by parts gives
$$R_1(u)=\frac{\int_\Omega |Du|^2 dx}{\int_\Omega |u|^2 dx}$$
which is the usual Rayleigh quotient for the Laplace operator  with $u\in W^{1,2}_0(\Omega)$. In this proof, all functions involved, including $u_m\leq 0$, belong to $W^{1,2}_0(\Omega)$ so this is the formula for $R_1(u)$ that we will use.

Recall that the first Laplace eigenvalue of $\Omega$ has the following variational characterization
$$\lambda(1;\Omega) =\inf\left\{R_1(u): u\in W^{1,2}_0(\Omega)\backslash \{0\}\right\}.$$
Since $u_m\leq 0$ in $\Omega$ for all $m$, we can rewrite (\ref{kIS}) as
\begin{equation}
\label{kIS1}
-\Delta u_{m+1}= R_1(u_m) u_m - (m+1)^{-2}\quad\text{in }\Omega,\quad u_m=0\quad\text{on }\p\Omega.
\end{equation}
By (\ref{umC}), we have for all $m\geq 1$,
\begin{equation*} 
\|u_m\|_{L^{2}(\Omega)}\leq C_1(u_0,\Omega).
\end{equation*}
As observed right before (\ref{SkL1}), we also have 
$
R_1(u_m)\leq C_3(u_0,\Omega)
$  for all $m\geq 1$.
Hence,
$$\int_\Omega \left(|Du_m|^2 + |u_m|^2\right) dx= [R_1(u_m) + 1]\|u_m\|^2_{L^{2}(\Omega)} \leq C_4(u_0,\Omega).$$
The sequence $\{u_m\}$ is uniformly bounded in $W^{1,2}_0(\Omega)$. Therefore, there is a subsequence $u_{m_j}$ such that $u_{m_j}$ converges weakly in $W^{1,2}_0(\Omega)$ and strongly in $L^2(\Omega)$ to $u_\infty\in W^{1,2}_0(\Omega)$ where $u_\infty\leq 0$.  As noticed in $(ii)$, we have $u_\infty\not\equiv 0$. 

From (\ref{kIS1}), we deduce that $u_{m_j+1}$ converges weakly in $W^{1,2}_0(\Omega)$ and strongly in $L^2(\Omega)$ to $w_\infty\in W^{1,2}_0(\Omega)$ where $w_\infty\leq 0$ and $w_\infty\not\equiv 0$.

Now, we let $j\rightarrow \infty$ in 
$$-\Delta u_{m_j+1}= R_1(u_{m_j}) u_{m_j} - (m_j+1)^{-2}\quad\text{in }\Omega,\quad u_{m_j}=0\quad\text{on }\p\Omega.
$$
 we obtain, as in $(iii)$, using (\ref{Rku_con}) that
\begin{equation}\label{weakweq}-\Delta w_{\infty}= \lambda(1;\Omega) u_\infty \quad\text{in }\Omega
\end{equation}
and
$$\|w_{\infty}\|_{L^{2}(\Omega)} = \|u_{\infty}\|_{L^{2}(\Omega)}. $$
The equation (\ref{weakweq}) is understood in the sense that: for all $\varphi\in W^{1,2}_0(\Omega)$, we have
$$\int_\Omega Dw_{\infty}\cdot D\varphi dx=\lambda(1;\Omega) \int_\Omega u_\infty \varphi dx.$$
In particular, 
$$\int_\Omega |Dw_\infty|^2 dx = \lambda(1;\Omega) \int_\Omega u_\infty w_\infty dx$$
Applying the H\"older inequality in
\begin{eqnarray*}R_1(w_\infty) \|w_{\infty}\|^{2}_{L^{2}(\Omega)}=\int_{\Omega} |Dw_{\infty}|^2 dx =\lambda(1; \Omega) \int_{\Omega} u_{\infty} w_{\infty} ~dx&\leq& 
\lambda(1; \Omega) \|u_{\infty}\|_{L^{2}(\Omega)}  \|w_{\infty}\|_{L^{2}(\Omega)}\\&=& \lambda(1; \Omega)\|w_{\infty}\|^{2}_{L^{2}(\Omega)}
\end{eqnarray*}
together  $R_1(w_\infty)\geq \lambda(1;\Omega)$, we find that the above inequality becomes an equality and we obtain a constant $c>0$ such that $u_{\infty}= c w_{\infty}$  and 
$$-\Delta w_{\infty} = c \lambda(1; \Omega) w_\infty\quad\text{in }\Omega.$$ Since $w_\infty\leq 0$, $w_\infty\not\equiv 0$, we deduce that $c\lambda(1;\Omega)$ is the first Laplace eigenvalue. Its 
 uniqueness then allows us to conclude that $c=1$ and $u_\infty=w_\infty$ is a first Laplace  eigenfunction of $\Omega$. 
 
 Using (\ref{Llim}) as in $(iii)$, we find that 
 whole sequence $u_m$ converges weakly in $W^{1,2}_0(\Omega)$ and strongly in $L^2(\Omega)$ to $u_\infty$. This convergence is strong in $W^{1,2}_0(\Omega)$. Indeed, by (\ref{kIS1}), we can write
 $$-\Delta (u_{m+1}- u_\infty)= R_1(u_m) (u_m-u_\infty) + [R_1(u_m)-\lambda(1;\Omega)] u_\infty -(m+1)^{-2}\quad\text{in }\Omega.$$
 Multiplying both sides by $u_{m+1}-u_\infty$ and integrating by parts, we easily conclude $\|D(u_{m+1}-u_\infty)\|^2_{L^2(\Omega)}\rightarrow 0$.
\end{proof}
In the proof of Theorem \ref{conRk}$(ii)$, we use the following estimate due to Trudinger-Wang.
\begin{thm} [\cite{TWk2}, Theorem 4.1]
\label{TWthm}
 Let $u\in C^2(\Omega)$ be $k$-admissible and satisfy $u\leq 0$ in $\Omega$. Then for any subdomain $V\subset\subset \Omega$ and all $q<\frac{nk}{n-k}$, we have the estimate
$$\int_V |Du|^q~dx\leq C(V,\Omega, n, k, q) \left(\int_\Omega |u|~dx\right)^q.$$
\end{thm}
\begin{rem}
(a)  Let $w\in C^{\infty}(\Omega)\cap C^{1,1}(\overline{\Omega})$ be a nonzero $k$-Hessian eigenfunction as in (\ref{kEVP_eq}). In view of Theorem \ref{conRk} (i), we deduce from (\ref{um1w}) and (\ref{upum1}) the following result for the scheme (\ref{kIS}):
\begin{equation} 
\label{Gdcon}
\lim_{m\rightarrow \infty} \left[\int_\Omega |u_{m+1}| S_k(D^2 w)dx -\int_\Omega |w| [S_k(D^2 w)]^{\frac{k-1}{k}} [S_k(D^2 u_{m+1})]^{\frac{1}{k}} dx\right]=0.
\end{equation}
Indeed, let $b_m$ be the difference in the above bracket. Then, $b_m\geq 0$ by (\ref{um1w}). As in (\ref{upum1}), we have
\begin{eqnarray*}
[\lambda(k;\Omega)]^{-k} b_m&=&\int_\Omega |u_{m+1}| |w|^k dx -  \int_\Omega [\lambda(k;\Omega)]^{-1} |w|^k \left[ R_k(u_m) |u_m|^k + (m+1)^{-2}\right]^{\frac{1}{k}} dx\\
&<& \int_\Omega |u_{m+1}| |w|^k dx -  \int_\Omega [\lambda(k;\Omega)]^{-1} [R_k(u_m)]^{\frac{1}{k}}|u_m||w|^k dx\rightarrow 0\text{ when } m\rightarrow\infty.
\end{eqnarray*}
In the last convergence, we used  (\ref{Llim}) and $ [\lambda(k;\Omega)]^{-1} [R_k(u_m)]^{\frac{1}{k}}\rightarrow 1$ as given by (\ref{Rku_con}).\\
(b) We can use \cite[Lemma 2.2]{TWk2} to show that the limit function $u_\infty$ in Theorem \ref{conRk} (ii) possesses certain convexity properties, called $k$-convexity in \cite{TWk2}.
\end{rem}

\begin{rem} Consider the case $2\leq k\leq n-1$. As remarked after the statement of Theorem \ref{conRk}, 
showing that $u_{\infty}$ in Theorem \ref{conRk} $(ii)$ is a $k$-Hessian eigenfunction is an interesting open problem. Moreover, we do not know how to prove the full convergence of $u_m$ to $u_\infty$ in some suitable sense as in the Monge-Amp\`ere case. In the Monge-Amp\`ere case, the uniqueness issue of the Monge-Amp\`ere eigenvalue problem (\ref{EVP_eq}) with $u$ only being convex is now well understood and this plays a key role in the proof of Theorem \ref{conRk} $(iii)$ as it was used to conclude that $c=1$, among other results.

For a $k$-convex function $u$, we can define a weak notion of its $k$-Hessian, still denoted by $S_k(D^2u)$,  (see, \cite{TWk2} for example). 
Consider the following degenerate $k$-Hessian equation
\begin{equation}
\label{SkEVPw}
S_k(D^2 w)= \lambda |w|^k\quad\text{in }\Omega, \quad w=0\quad\text{on }\p\Omega
\end{equation}
for a nonzero $k$-convex function $w$ and a positive constant $\lambda>0$.  
To the best of the author's knowledge, the following questions concerning (\ref{SkEVPw}) are still open:
\begin{enumerate}
\item[(i)] Is $w$ smooth in $\Omega$? 
\item[(ii)] Is $\lambda$ unique? 
\item[(iii)] Is $w$ unique up to a positive multiplicative constant?
\end{enumerate}
In the Monge-Amp\`ere case, the answers to all these questions are positive in
 \cite{LSNS} which relies on the regularity theory
of weak solutions to the Monge-Amp\`ere equation developed by Caffarelli \cite{C1, C2}.  The $k$-Hessian counterparts of these Monge-Amp\`ere results are still lacking.
\end{rem}

 \section{$W^{2,1}$ convergence for the non-degenerate inverse iterative scheme}
 \label{W21sect}
In this section, we prove Theorem \ref{W21k}.
\begin{proof}[Proof of Theorem \ref{W21k}]
Recall that we are considering the case $2\leq k\leq n$ for the scheme (\ref{kIS}).  
Integrating by parts as in (\ref{ibpuv}), we have
\begin{equation}
\label{ibpuv2}
\int_\Omega |u_{m+1}| S_k(D^2 w)dx = \int_\Omega \frac{1}{k}|w|S_k^{ij}(D^2 w) D_{ij} u_{m+1} dx.
\end{equation}
Consider the following hyperbolic polynomial as defined in (\ref{PPk}) $$p(A)=P_k(A) \quad \text{where } \lambda (A)\in\Gamma_k.$$  Recall the notation $\lambda_i(P_k; A, X)$ in Section \ref{Hypersect} and $\Gamma_k$ as in (\ref{Gak}). To simplify the notation, we denote
$$\lambda_{k, i}(A, X)= \lambda_i(P_k; A, X).$$
When $k=n$, $\lambda_{k,i}(A,X)$'s are all eigenvalues of  the matrix $XA^{-1}$. Let $$A= D^2 w,\quad X_m= D^2 u_{m+1}.$$
Then
$$S_k^{ij}(D^2 w) D_{ij}u_{m+1}= p'_{X_m}(A) = \sum_{i=1}^k \lambda_{k, i}(A, X_m) p(A)\quad\text{and } \frac{p(X_m)}{p(A)}=\prod_{i=1}^k \lambda_{k, i}(A, X_m)$$
and
$$ \left[\frac{p(X_m)}{p(A)}\right]^{1/k} p(A) = \left[p(X_m)\right]^{\frac{1}{k}} [p(A)]^{\frac{k-1}{k}} = \left[S_k(D^2 w)\right]^{\frac{k-1}{k}}  \left[S_k(D^2 u_{m+1})\right]^{\frac{1}{k}}.$$
Using (\ref{QGar}), we find that 
\begin{multline}
\label{Gdconm}
\frac{1}{k}S_k^{ij}(D^2 w) D_{ij}u_{m+1} =\frac{1}{k}  \sum_{i=1}^k \lambda_{k, i}(A, X_m) p(A) = \frac{1}{k} p'_{X_m}(A) p(A) \\
\geq [S_k(D^2 w)]^{\frac{k-1}{k}}  [S_k(D^2 u_{m+1})]^{\frac{1}{k}} + \frac{1}{k}\sum_{i=1}^k\left(\sqrt{\lambda_{k, i}(A, X_m)}-\left[\frac{p(X_m)}{p(A)}\right]^{\frac{1}{2k}}\right)^2 p(A).
\end{multline}
By combining (\ref{ibpuv2}), (\ref{Gdconm}) and (\ref{Gdcon}), we deduce that
\begin{equation}
\label{W21w}
\int_{\Omega} |w| \sum_{i=1}^k\left(\sqrt{\lambda_{k, i}(A, X_m)}-\left[\frac{p(X_m)}{p(A)}\right]^{\frac{1}{2k}}\right)^2 p(A) dx\rightarrow 0\quad\text{when } m\rightarrow\infty.
\end{equation}
From the uniform $L^1(\Omega)$ bound for $u_m$ which can be derived from (\ref{umC}), and (\ref{ibpuv2}), we get
$$\int_\Omega  \sum_{i=1}^k \lambda_{k, i}(A, X_m) p(A) |w|dx\leq C(k,u_0,\Omega).$$
Since $|w|\geq c(V)>0$ for each $V\subset\subset \Omega$, and $p(A)=S_k(D^2 w)= [\lambda(k;\Omega)]^k|w|^k,$ we obtain that
\begin{equation}
\label{L1bdk}
\int_V  \sum_{i=1}^k\lambda_{k, i}(A, X_m)dx \leq C(V)\quad \text{for each } V\subset\subset \Omega.
\end{equation}
Thus, (\ref{W21w}) and (\ref{L1bdk}) imply the following convergence 
\begin{equation}
\label{L1AXm}
\lambda_{k, i}(A, X_m)- \left[\frac{p(X_m)}{p(A)}\right]^{1/k}\rightarrow 0\quad\text{locally in } L^1 \text{ when } m\rightarrow\infty.
\end{equation}
To see this, let $V\subset\subset\Omega$ be a non-empty open set. Then $$[p(A)]^{\frac{1}{k}}= \lambda(k;\Omega) |w| \geq \lambda (k;\Omega) c(V)= c_5(k, \Omega, V)>0.$$ Thus, 
(\ref{W21w}) implies that
\begin{equation}
\label{W21w2}
\int_{\Omega} \sum_{i=1}^k\left(\sqrt{\lambda_{k, i}(A, X_m)}-\left[\frac{p(X_m)}{p(A)}\right]^{\frac{1}{2k}}\right)^2 dx\rightarrow 0\quad\text{when } m\rightarrow\infty.
\end{equation}
By the H\"older inequality and (\ref{SkL1}), we find
\begin{eqnarray}
\label{L1pXmA} \int_V \left[\frac{p(X_m)}{p(A)}\right]^{1/k} dx &\leq& \frac{1}{c_5(k, \Omega, V)} \int_V [S_k(D^2 u_{m+1})]^{\frac{1}{k}} dx\nonumber \\&\leq&  \frac{|V|^{\frac{k-1}{k}}}{c_5(k, \Omega, V)} \left(\int_V S_k(D^2 u_{m+1}) dx\right)^{\frac{1}{k}} \leq C_6(k, u_0,\Omega, V).
\end{eqnarray}
Again, by the H\"older inequality, we have 
\begin{multline*}\displaystyle
\left(\int_V\left|\lambda_{k, i}(A, X_m)- \left[\frac{p(X_m)}{p(A)}\right]^{1/k}\right| dx\right)^2\\
\leq \int_V \left(\sqrt{\lambda_{k, i}(A, X_m)}-\left[\frac{p(X_m)}{p(A)}\right]^{\frac{1}{2k}}\right)^2 dx  \int_V \left(\sqrt{\lambda_{k, i}(A, X_m)}+\left[\frac{p(X_m)}{p(A)}\right]^{\frac{1}{2k}}\right)^2 dx\\
\leq 2  \int_V \left(\sqrt{\lambda_{k, i}(A, X_m)}-\left[\frac{p(X_m)}{p(A)}\right]^{\frac{1}{2k}}\right)^2dx\int_V \left( \lambda_{k, i}(A, X_m) + \left[\frac{p(X_m)}{p(A)}\right]^{1/k} \right) dx\\
\leq 2 (C(V) + C_6(k, u_0,\Omega, V)) \int_V \left(\sqrt{\lambda_{k, i}(A, X_m)}-\left[\frac{p(X_m)}{p(A)}\right]^{\frac{1}{2k}}\right)^2dx.
\end{multline*}
In the last estimate, we used  (\ref{L1bdk}) and (\ref{L1pXmA}). Now, letting $m\rightarrow\infty$ in the above inequality and recalling (\ref{W21w2}), we obtain (\ref{L1AXm}) as claimed.

Recall from parts $(i)$ and $(ii)$ of Theorem \ref{conRk} that, when $j\rightarrow\infty$,
\begin{multline}
\label{L1AXm2}
\frac{p(X_{m_j})}{p(A)} = \frac{S_k(D^2 u_{m_j+1})}{S_k(D^2 w)} =\frac{R_k(u_{m_j}) |u_{m_j}|^k + (m_j+1)^{-2}}{[\lambda(k;\Omega)]^k |w|^k }\\ \rightarrow \frac{[\lambda(k;\Omega)]^k |u_\infty|^k}{[\lambda(k;\Omega)]^k|w|^k}=\frac{|u_\infty|^k}{|w|^k} \quad\text{locally in } L^1.
\end{multline}
In the above local $L^1$ convergence, as in (\ref{L1AXm}), we also use that $|w|$ has a positive lower bound on each compact subset of $\Omega$.

It follows from (\ref{L1AXm}) and (\ref{L1AXm2}) that 
$$\lambda_{k, i}(D^2 w, D^2 u_{m_j+1})=\lambda_{k, i}(A, X_{m_j}) \rightarrow \frac{|u_\infty|}{|w|} \quad\text{locally in } L^1 \text{ when } j\rightarrow\infty.$$
Finally, we prove the pointwise convergence (\ref{conpt}).
The above local $L^1$ convergence shows that, up to extracting a subsequence, still denoted $(u_{m_j})$, we have
 $$\lambda_{k, i}(A(x), X_{m_j}(x)) \rightarrow \frac{|u_\infty(x)|}{|w(x)|} \quad\text{a.e. }x\in\Omega\quad\text{for all } i=1,\cdots, k.$$
Thus, we deduce from (\ref{lameq}) that
\begin{equation}
\label{lamki0}
\lambda_{k, i}\left(A(x), X_{m_j}(x)-\frac{|u_\infty(x)|}{|w(x)|} A(x) \right) \rightarrow 0\quad\text{a.e. }x\in\Omega\quad\text{for all } i=1,\cdots, k.
\end{equation}
Since $k\geq 2$, we know from Lemma \ref{E0lem} that $E_{A(x)}(P_k)=\{0\}$ for all  $x\in\Omega.$ It follows from (\ref{lamki0}) and Lemma \ref{lambm} that
$$D^2 u_{m_j+1}(x)-\frac{|u_\infty(x)|}{|w(x)|} D^2 w(x) \rightarrow 0\quad\text{a.e. }x\in\Omega.$$
Therefore, we have (\ref{conpt}) and the proof of our theorem is complete.
\end{proof}
\begin{rem}
The $L^1$ convergence  in Theorem \ref{W21k} uses the fact that when $k\geq 2$, the second sum in (\ref{QGar}) gives nontrivial information.
When $k=1$, this sum is $0$; however,
 by Theorem \ref{conRk} $(iv)$, we have $u_\infty= cw$ for some constant $c>0$, and the following full convergence 
\begin{equation*}\lambda_{1,1}(A, X_m) =\frac{\Delta u_{m+1}}{\Delta w}= \frac{R_1(u_m) u_m - (m+1)^{-2}}{\lambda(1;\Omega) w} \rightarrow \frac{u_\infty}{w}=c\quad\text{locally in } W^{1,2}.
\end{equation*}
\end{rem}
Now, consider the scheme (\ref{kIS}) with $k=n$. Then, $\lambda_{n, 1}(A, X_m)$, $\cdots$, $\lambda_{n, n}(A, X_m)$ in Theorem \ref{W21k} are eigenvalues of $D^2 u_{m+1} (D^2 w)^{-1}$. 
From Theorem \ref{conRk} $(iii)$, we know that $u_m$ converges uniformly to a nonzero Monge-Amp\`ere eigenfunction $u_\infty$ which is a positive multiple of $w$.
Without loss of generality, we can assume that $u_\infty=w$. Thus, Theorem \ref{W21k} shows 
 that $$D^2 u_{m+1} (D^2 w)^{-1}\rightarrow I_n\quad \text{locally in }L^1(\Omega)$$ and hence  $D^2 u_{m+1} \rightarrow D^2 w$ locally in $L^1(\Omega)$.
 Thus a rigidity form of Proposition \ref{ibplem}, that is (\ref{QGar}) of Lemma \ref{Glem},  improves the uniform convergence of $u_m$ to $w$ to an interior $W^{2,1}$ convergence. Note that this $W^{2,1}$ convergence also follows from the general result in De Philippis-Figalli \cite[Theorem 1.1]{dPF} but our proof here is different and it also works for the $k$-Hessian eigenvalue problem. We state this convergence in the following theorem.
\begin{thm}
\label{W21n}
Let $\Omega$ be a bounded, open, smooth and uniformly convex domain in $\R^n$. 
Let $k=n$ and let $w$ be a nonzero Monge-Amp\`ere eigenfunction of $\Omega$ to which the solution $u_{m}$ of (\ref{kIS}) converges uniformly. Then $D^2 u_{m}$ converges locally in $L^1$ to $D^2 w$ in $\Omega$.
\end{thm}

\begin{rem}
Hidden in the variational characterizations (\ref{lam_def1}) and (\ref{klamR}) of the Monge-Amp\`ere and $k$-Hessian eigenvalues via the Rayleigh quotients defined in (\ref{RQk}) is the divergence form of the $k$-Hessian operators. For $k=1$, the frequently used Rayleigh quotient is
$$Ra(u) = \frac{\int_{\Omega} |Du|^2~dx}{\int_{\Omega} |u|^{2}~dx}.$$ When $u\in C^2(\overline{\Omega})$ with $u\leq 0$ in $\Omega$ and $u=0$ on $\p\Omega$, $Ra(u)$ is equal to $R_1(u)$ (defined in (\ref{RQk})) due to a simple integration by parts. Thus the divergence form of $S_1(D^2 u)=\Delta u$ is used here.
For non-divergence form operators such as $\mathcal{M}_{n-1}(D^2 u)$ in (\ref{MpMA}), we do not expect their first eigenvalues (if any) to have a variational characterization as the $k$-Hessian eigenvalues.
However, we expect the spectral characterizations of the $k$-Hessian eigenvalues in Theorem \ref{kHessL} to have counterparts in purely non-divergence form operators generated by hyperbolic polynomials. 
\end{rem}
{\bf Acknowledgements.} The author warmly thanks the referees for their careful reading of the original version of the paper. Their constructive comments and suggestions help improve the paper.

\end{document}